\documentclass[12pt]{article}

\usepackage{amsfonts}
\usepackage{tikz}
\usepackage{pgfplots}

\newtheorem{theorem}{Theorem}

\newtheorem{corollary}[theorem]{Corollary}

\newtheorem{lemma}[theorem]{Lemma}

\newenvironment{proof}[1][Proof]{\noindent\textbf{#1.} }{\ \rule{0.5em}{0.5em}}
\begin{document}

\title{Random Euclidean embeddings in finite dimensional Lorentz
spaces}
\author{Daniel J. Fresen\thanks{%
University of Pretoria, Department of Mathematics and Applied Mathematics,
daniel.fresen@up.ac.za MSC: 46B06, 46B07, 46B09, 52A21, 52A23.}}
\maketitle

\begin{abstract}
Quantitative bounds for random embeddings of $\mathbb{R}^{k}$ into Lorentz
sequence spaces are given, with improved dependence on $\varepsilon$.
\end{abstract}






\section{Introduction}

Our starting point is Milman's general Dvoretzky theorem \cite{Mil}. The dependence on $\varepsilon$ is due to Schechtman \cite{Sch} following Gordon \cite{Gord}, and the Gaussian formulation due to Pisier \cite{Pis}. We refer the reader to \cite{MS, PVZ, Sch2} for more details.
\begin{theorem}
\label{LoMilGen}There exists a universal constant $c>0$ such that the following is true. Consider any
\[
(n,k,\varepsilon)\in \mathbb{N}\times\mathbb{N}\times (0,1)
\]
Let $\left\Vert \cdot \right\Vert$ be a norm on $\mathbb{R}^n$ and set
\[
M=\left(2\pi\right)^{-\frac{n}{2}}\int_{\mathbb{R}^n}\left\Vert x \right\Vert \exp\left({-\frac{1}{2}\left\vert x \right\vert^2}\right)dx \hspace{2cm} b=\sup\left\{\left\Vert \theta \right\Vert :\theta\in S^{n-1}\right\}
\]
where $\left\vert \cdot \right\vert$ denotes the standard Euclidean norm. Assume that $k\leq d$ where
\[
d= c\left(\frac{M}{b}\right)^2\varepsilon^2
\]
and let $G$ be an $n\times k$ random matrix with i.i.d. standard normal random variables as entries. With probability at least $1-2e^{-d}$, the following event occurs: for all $x\in \mathbb{R}^k$,
\begin{eqnarray}
(1-\varepsilon)M\left\vert x \right\vert \leq \left\Vert Gx \right\Vert \leq (1+\varepsilon)M\left\vert x \right\vert \label{epsilon iso}
\end{eqnarray}
\end{theorem}

If $\left(X,\left\Vert\cdot\right\Vert\right)$ is a general $n$-dimensional normed space over $\mathbb{R}$, one can put coordinates on $X$ by identifying it with $\mathbb{R}^n$ in such a way so that
\[
b=1 \hspace{3cm} M\geq c\sqrt{\ln n}\label{parm}
\]
and therefore if $n$ is sufficiently large we can ensure that $d$ is large. Logarithmic dependence on $n$ is the worst case scenario (for the best possible choice of coordinates), which is the correct behavior in $\ell_\infty^n$ but can be improved significantly for other spaces, such as $\ell_p^n$ for fixed $p\in\left[1,\infty\right)$, where one gets power dependence on $n$.

By rotational invariance of the normal distribution, $\mathrm{Range}(G)$ is a uniformly distributed random subspace in the Grassmannian $G_{n,k}$ of all $k$-dimensional linear subspaces of $\mathbb{R}^n$. The bounds in (\ref{epsilon iso}) mean that on $\mathrm{Range}(G)$, $\left\Vert \cdot \right\Vert$ approximates the pushforward norm $\left\Vert y \right\Vert_{\sharp}:=M\left\vert G^{-1}y \right\vert$, where $G^{-1}:\mathrm{Range}(G)\mapsto \mathbb{R}^k$ denotes the inverse of the linear map associated to $G$. Since the sub-level sets of this pushforward norm are ellipsoids, Milman's general Dvoretzky theorem can be interpreted as follows: Assuming $n$ is sufficiently large and we have chosen an appropriate coordinate system through which to identify a given normed space $X$ with $\mathbb{R}^n$, 

\medskip

\noindent $\bullet$ Most $k$-dimensional subspaces of $X$ are almost isometric to Hilbert spaces

\noindent $\bullet$ Most $k$-dimensional cross-sections of the unit ball $B=\left\{x:\left\Vert x \right\Vert \leq 1\right\}$ are approximately ellipsoidal.

\medskip

For $\varepsilon=\varepsilon_0$, for any universal constant $\varepsilon_0\in (0,1)$, Theorem \ref{LoMilGen} is in a particular sense sharp (see \cite{HuWe, MS2}), and in this sense dependence on $n$ is understood. However the question of optimal dependence on $\varepsilon$ is open. We refer the reader to \cite{PaVa, Sch3} for the best existing bounds of the form $\rho\left(\varepsilon\right)\ln n$ in the existential Dvoretzky theorem, where one is satisfied with a single subspace of this dimension, and bounds for general classes of spaces with symmetries are contained in \cite{Fr0, Tik}.

\medskip

\hspace*{-0.92cm}\noindent
\fbox{
    \parbox{\textwidth}{Paouris, Valettas and Zinn \cite{PVZ} studied dependence on $\varepsilon$ in the randomized Dvoretzky theorem for the $\ell_p^n$ spaces, both in the range $1\leq p\leq C\ln n$ and $p>C\ln n$, improving on the bounds in Theorem \ref{LoMilGen}.

The results we now present extend results in \cite{PVZ} to the class of Lorentz spaces. These spaces have a structure that is more complicated than that of the $\ell_p^n$ spaces, and the Lorentz norm of a Gaussian random vector is typically not written in terms of the sum of i.i.d. random variables. Our approach is different to that in \cite {PVZ} and in the special case of $\ell_p^n$ it allows for a simpler proof without the use of Talagrand's $L_1-L_2$ inequality, and with improved dependence on $p$ (removing a factor of $p^{-p}$).

    }
}

\section{What's new?}

Consider the finite dimensional Lorentz spaces, i.e. $\mathbb{R}^n$ endowed with the norm
\[
\left\vert x\right\vert_{\omega,p}=\left(\sum_{i=1}^n\omega_i x_{[i]}^p\right)^{1/p}
\]
where $1\leq p <\infty$, $\left(\omega_i\right)_1^n$ is any non-increasing sequence in $[0,1]$ with $\omega_1=1$, and $\left( x_{\left[ i\right] }\right) _{1}^{n}$ denotes the
non-increasing rearrangement of $\left( \left\vert x_{i}\right\vert \right)
_{1}^{n}$. These spaces are the finite dimensional counterparts to the infinite dimensional Lorentz spaces (see \cite{LiTz1}) which play a classical role in analysis.

Before studying the general case in Section \ref{LoGen}, we study the special case where $\omega_i=i^{-r}$ for $0\leq r<\infty$, using the notation
\[
\left\vert x\right\vert _{r,p}=\left( \sum_{i=1}^{n}i^{-r}x_{\left[ i\right]
}^{p}\right) ^{1/p} 
\]%
In this special case, using Lemmas \ref{Lo med calculation} and \ref{Lo Lip con}, Theorem \ref{LoMilGen} applies with
\[
d= \left\{ 
\begin{array}{ccccc}
c_{r,p}n\varepsilon ^{2} & : & 0\leq r \leq 1/2 , & p<2\left( 1-r\right) \\ 
c_{r,p}n\left( \ln n\right) ^{1-\frac{2}{p}}\varepsilon ^{2} & : & 0\leq r \leq 1/2 , & p=2\left( 1-r\right) \\ 
c_{r,p}n^{\frac{2(1-r)}{p}}\varepsilon ^{2} & : & 0\leq r \leq 1/2 , & p>2\left( 1-r\right) \\ 
c_{r,p}n^{\frac{2(1-r)}{p}}\varepsilon ^{2} & : & 1/2<r<1\\ 
c_{r,p}\left( \ln n\right) ^{1+\frac{2}{p}}\varepsilon ^{2} & : & r=1 &  & 
\end{array}%
\right. 
\]
where the coefficients $c_{r,p}>0$ do not depend on anything except $r$ and $p$ and can be written explicitly. Our first main result, Theorem \ref{LoMain}, is a little hard on the eye, and we have hidden it near the end of the paper. Its two main corollaries, however, are simpler (Corollaries \ref{Corrasy} and \ref{Corrp}). Both improve the dependence on $\varepsilon$.

\begin{corollary}
\label{Corrasy}In the case $\left\Vert \cdot \right\Vert=\left\vert \cdot \right\vert _{r,p}$, where $0\leq r \leq 1$ and $1\leq p<\infty$, Theorem \ref{LoMilGen} holds with the sufficient condition $k\leq d$ replaced with $k\leq d'$, where
\[
d'\geq \left\{ 
\begin{array}{ccc}
c_{r,p}n\varepsilon^2 & : & 0\leq r<1/2, p<2-2r\\ 
c_{r,p}\min\left\{n\varepsilon^2,n\left(\ln n\right)^{1-\frac{2}{p}}\varepsilon^{\frac{2}{p}}\right\} & : & 0\leq r<1/2, p=2-2r\\
c_{r,p}\min\left\{n\varepsilon^2,n^{\frac{2(1-r)}{p}}\varepsilon^{\frac{2}{p}}\right\} & : & 0\leq r<1/2, p>2-2r\\
c_{r,p}n\left(\ln n\right)^{-1}\varepsilon^2 & : & r=1/2, p=2-2r=1\\
c_{r,p}\min\left\{n\left(\ln n\right)^{-p}\varepsilon^2,n^{\frac{1}{p}}\varepsilon^{\frac{2}{p}}\right\} & : & r=1/2, p>2-2r=1\\
c_{r,p}\min\left\{n^{2(1-r)}\left(\ln n\right)^{-(p-1)}\varepsilon^2,n^{\frac{2(1-r)}{p}}\varepsilon^{\frac{2}{p}}\right\} & : & 1/2<r<1\\
c_{r,p}\min\left\{\left(\ln n\right)^{3}\varepsilon^2,\left(\ln n\right)^{1+\frac{2}{p}}\varepsilon^{\frac{2}{p}}\right\} & : & r=1
\end{array}%
\right.
\]
\end{corollary}

For the spaces under consideration, Corollary \ref{Corrasy} improves on Theorem \ref{LoMilGen} in all cases except when either $p<2-2r$ or when $p=1$, and in those cases it reduces to the old bound.

We have excluded the case $r>1$ from Corollary \ref{Corrasy} because as $n\rightarrow \infty$, the resulting Lorentz space is isomorphic to $\ell_\infty^n$, and Theorem \ref{Lo ell infty regime} below gives better estimates. More generally, when
\[
p>c\ln \left( 1+\frac{1+n^{1-r}}{1+\left\vert
1-r\right\vert \ln n}\ln n\right)  
\]%
for arbitrarily small universal constant $c>0$, Lemma \ref{Lo basic power int} (later in the paper) implies%
\[
\left\vert x\right\vert _{\infty }\leq \left\vert x\right\vert _{r,p}\leq
C\left\vert x\right\vert _{\infty }\left( 1+\frac{1+n^{1-r}}{1+\left\vert
1-r\right\vert \ln n}\ln n\right) ^{1/p}\leq C^{\prime }\left\vert
x\right\vert _{\infty } 
\]%
In this case the estimates in Theorem \ref{LoMain} start to break down. The
following result, which generalizes the case $p>c\ln n$ (and $r=0$) in \cite[Theorem 1.2]%
{PVZ} can then be used instead (with slightly improved probability bound from $1-Cn^{-c\varepsilon/\ln (1/\varepsilon)}$ in \cite{PVZ}).

\begin{theorem}
\label{Lo ell infty regime}For all $0<c_{1}<C_{1}$ there exists $c_{2}>0$
such that the following is true: let $n\in \mathbb{N}$ and let $\left\vert
\cdot \right\vert _{\sharp }$ be a norm on $\mathbb{R}^{n}$ that is
invariant under coordinate permutations and satisfies $c_{1}\left\vert
x\right\vert _{\infty }\leq \left\vert x\right\vert _{\sharp }\leq
C_{1}\left\vert x\right\vert _{\infty }$ (for all $x\in \mathbb{R}^{n}$).
Let $\varepsilon \in \left( 0,1\right) $ and $0<k\leq c_{2}\varepsilon
\left( \ln \varepsilon ^{-1}\right) ^{-1}\ln n$. Let $G$ be an $n\times k$
standard Gaussian random matrix. Then with probability at least $%
1-Cn^{-c_{2}\varepsilon }$, for all $x\in \mathbb{R}^{k}$, $\left(
1-\varepsilon \right) \mathbb{M}\left\vert Ge_{1}\right\vert _{\sharp
}\left\vert x\right\vert \leq \left\vert Gx\right\vert _{\sharp }\leq \left(
1+\varepsilon \right) \mathbb{M}\left\vert Ge_{1}\right\vert _{\sharp
}\left\vert x\right\vert $.
\end{theorem}

\begin{proof}
Let $T:\mathbb{R}^{n}\rightarrow \mathbb{R}^{n}$ be the map that arranges
the coordinates of a vector in non-decreasing order. Let $X$ and $Y$ be
independent standard normal random vectors in $\mathbb{R}^{n}$. From
estimates for the normal distribution, see e.g. (\ref{ellinftyest}) in Lemma \ref{normalorderstats}, with probability at least $1-C\exp \left( -ct^{2}\right) $,%
\[
\left\vert \left\vert X\right\vert _{\sharp }-\left\vert Y\right\vert
_{\sharp }\right\vert =\left\vert \left\vert TX\right\vert _{\sharp
}-\left\vert TY\right\vert _{\sharp }\right\vert \leq \left\vert
TX-TY\right\vert _{\sharp }\leq C_{1}\left\vert TX-TY\right\vert _{\infty
}\leq \frac{CC_{1}t^{2}}{\sqrt{\ln n}} 
\]%
This can be converted to a deviation of $\left\vert X\right\vert _{\sharp }$ about $\mathbb{M}\left\vert X\right\vert _{\sharp }$ and is the same deviation estimate satisfied by $\left\vert X\right\vert
_{\infty }$. The usual proof of the randomized Dvoretzky theorem (using
the $\left( \varepsilon /4\right) $-net argument, see e.g. \cite{MS, Sch2})
then transfers to $\left\vert \cdot \right\vert _{\sharp }$.
\end{proof}

The second corollary of Theorem \ref{LoMain} applies to the classical $\ell_p^n$ spaces, with improved dependence on $p$ compared to \cite[Theorem 1.2]{PVZ}.

\begin{corollary}
\label{Corrp}For all $C_1>0$ there exists $c_2>0$ such that the following statement is true. In the case $\left\Vert \cdot \right\Vert=\left\vert \cdot \right\vert _{0,p}=\left\vert \cdot \right\vert _{p}$, and under the added assumption that $p<C_1\ln n$, Theorem \ref{LoMilGen} holds with the sufficient condition $k\leq d$ replaced with $k\leq d'$, where $d'$ is defined as follows,
\[
d'=\left\{ 
\begin{array}{ccccc}
cn\varepsilon^2 & : &1\leq p\leq2 \\ 
c_2\min\left\{c^pn\varepsilon^2,pn^{\frac{2}{p}}\varepsilon^{\frac{2}{p}} \right\} & : & 2<p<C_1\ln n
\end{array}%
\right. 
\]
and $c>0$ is a universal constant.
\end{corollary}

We end the paper with a result for general Lorentz norms.

\begin{theorem}
(refer to Theorem \ref{LoMain2} for a more detailed statement) A random embedding of $\left(\mathbb{R}^k,\left\vert\cdot\right\vert\right)$ into $\left(\mathbb{R}^n,\left\vert\cdot\right\vert_{\omega,p}\right)$ using a standard Gaussian matrix is, with probability at least $1-2e^{-d}$, an $\varepsilon$ almost isometry in a sense similar to (\ref{epsilon iso}) provided $k\leq d$, where $d$ is defined as follows: If $1<p<\infty$ set
\[
d=\left(1+\frac{1}{p-1}\right)^{-1}\min \left\{\frac{c^p\left( \sum_{i=1}^{n}\omega_i\left(\ln \frac{n}{i}\right)^{p/2}\right) ^{2}\varepsilon^2}{\sum_{i=1}^{n}\omega_i^2\left(\ln \frac{n}{i}\right)^{p-1}},
cB^{-1/p}\left( \sum_{i=1}^{n}\omega_i\left(\ln \frac{n}{i}\right)^{p/2}\right) ^{2/p}\varepsilon^{2/p}\right\}
\]
where
\[
B=\left\{ 
\begin{array}{ccc}
\sum_{i=1}^n\omega_i^2i^{-(p-1)}
 & : & 1< p <3/2\\ 
\left(\sum_{i=1}^n\omega_i^{\frac{2}{2-p}}\right)^{2-p} & : & 3/2\leq p<2\\
1 & : & 2\leq p<\infty
\end{array}%
\right. 
\]
and if $p=1$ set
\[
d=\frac{c\left( \sum_{i=1}^{n}\omega_i\left(\ln \frac{n}{i}\right)^{1/2}\right) ^{2}\varepsilon^2}{\sum_{i=1}^{n}\omega_i^2}
\]
\end{theorem}

\section{\label{engine}The main engine: Gaussian concentration}

We make extensive use of the classical Gaussian concentration inequality. The simple proof of Maurey and Pisier is contained in \cite{Pis}.
\begin{theorem}
\label{clasgaus}Let $f:\mathbb{R}^n\rightarrow \mathbb{R}$ and let $X$ be a random vector in $\mathbb{R}^n$ with the standard normal distribution. Then for all $t\geq0$,
\[
\mathbb{P}\left\{\left\vert X-\mathbb{M}f(X)\right\vert>Ct\mathrm{Lip}(f) \right\}\leq 2\exp\left(-t^2\right)
\]
Assuming $\textrm{Lip}(f)<\infty$, the same holds true with $\mathbb{M}$ replaced by $\mathbb{E}$.
\end{theorem}

 If $\left\Vert\cdot\right\Vert:\mathbb{R}^n\rightarrow[0,\infty)$ is a norm, then
\[
\textrm{Lip}\left(\left\Vert\cdot\right\Vert\right)=\sup \left\{\left\Vert\theta\right\Vert:\theta \in S^{n-1}\right\} 
\]
which can be seen by applying the triangle inequality. Denoting this supremum as $b\left(\left\Vert \cdot \right\Vert\right)$, Gaussian concentration implies that with probability at least $1-2\exp\left(-t^2\right)$,
\[
\mathbb{M}\left\Vert X \right\Vert-Ctb\left(\left\Vert \cdot \right\Vert\right)\leq\left\Vert X \right\Vert \leq \mathbb{M}\left\Vert X \right\Vert+Ctb\left(\left\Vert \cdot \right\Vert\right)
\]
For our purposes the right hand inequality $\left\Vert X \right\Vert \leq \mathbb{M}\left\Vert X \right\Vert+Ctb\left(\left\Vert \cdot \right\Vert\right)$ will be sufficient. The following result due to Schechtman\cite{Sch} provides a uniform bound over a sphere rather than just at a point and reduces to Theorem \ref{clasgaus} in the case $k=1$.
\begin{theorem}
\label{Schechtbound}Let $f:\mathbb{R}^n\rightarrow \mathbb{R}$ and let $G$ be an $n\times k$ random matrix with i.i.d standard normal random variables as entries. Let $t\geq0$ and assume that $k\leq ct^2$. Then with probability at least $1-2\exp\left(-ct^2\right)$ the following event occurs: For all $\theta\in S^{k-1}$,
\[
\left\vert f\left(G\theta\right)-\mathbb{M}f\left(G\theta\right) \right\vert\leq t\textrm{Lip}(f)
\]
\end{theorem}

\section{Key methods}

We will deal with functions that may not be Lipschitz, or whose Lipschitz constant is not representative of the typical behaviour of the function. For such functions we will need to prove deviation inequalities for $f(X)$ about $\mathbb{M}f(X)$, where $X$ is a standard normal random vector as in Section \ref{engine}. In order to do so, it will be useful to restrict the function $f$ to a set $K$ with the following two properties:

\medskip

$\bullet$ $\textrm{Lip}\left(f|_K\right)$ is nicely bounded

$\bullet$ $\mathbb{P}\left\{X\notin K\right\}$ is small (here $X$ is normally distributed as in Theorem \ref{clasgaus})

\medskip

Exactly how one interprets `nicely bounded' and `small' may depend on the situation, and the reader will see the details in the proofs of Theorems \ref{LoMain} and \ref{LoMain2}. It follows from elementary metric space theory that $f|_K$ can be extended to a function $F:\mathbb{R}^n\rightarrow \mathbb{R}$ such that $\textrm{Lip}(F)=\textrm{Lip}\left(f|_K\right)$. One can then apply Gaussian concentration to $F$ and transfer the result back to $f$ since $f(X)=F(X)$ with high probability. This procedure appears in Bobkov, Nayar and Tetali \cite{BNT} and is further explored in \cite{Fr}.

The original way of applying concentration of measure to prove Dvoretzky's theorem, as in the classical works of Milman and Schechtman, e.g. \cite{Mil, MS, Sch}, was to study concentration of $\left\Vert X\right\Vert$ directly using Lipschitz properties of $\left\Vert \cdot\right\Vert$ (where $\left\Vert \cdot\right\Vert$ is the norm of the space in question). One actually considered a random point on the sphere as opposed to a Gaussian random vector, the Gaussian approach being made popular by Pisier \cite{Pis}, but the point is that regardles of the randomness used, the function under consideration was the norm. A trick that will be useful in the context of $\left\Vert \cdot\right\Vert=\left\vert \cdot\right\vert_{\omega,p}$ is to study concentration of $\left\vert X\right\vert_{\omega,p}^p=\sum \omega_iX_{[i]}^p$ instead of $\left\vert X\right\vert_{\omega,p}$ (using Gaussian concentration and the procedure from \cite{BNT} just mentioned), and then to convert the result back to a bound on $\left\vert X\right\vert_{\omega,p}$ by transforming the distribution under the action of $s\mapsto s^{1/p}$. As functions on $\mathbb{R}^n$, $\left\vert \cdot\right\vert_{\omega,p}$ and $\left\vert \cdot\right\vert_{\omega,p}^p$ are fundamentally different in terms of their local-Lipschitz properties: the first function achieves its Lipschitz constant on any neighbourhood of the origin, while for the second function, the problematic points where the norm of the gradient is large have been moved far away from the origin so that a convexity argument can be used for points within a certain convex body containing the origin.

So, we will need to bound the distribution of a gradient, which comes down to bounding the expression $\sum i^{-2r}X_{[i]}^{2(p-1)}$. Since the terms of this sum are not independent, one cannot use the classical theory of sums of independent random variables. For $p\geq3/2$ one can write such a quantity in terms of a norm and use Gaussian concentration applied to norms. For $1\leq p<3/2$ one cannot write the gradient as a function of a norm, and one therefore cannot use the equation $\textrm{Lip}\left(\left\Vert\cdot\right\Vert\right)=b\left(\left\Vert \cdot \right\Vert\right)$. This causes difficulties, but one can bound the gradient above by a function of a norm, which is an interesting problem in its own right, especially in the case $p=2-2r$. We postpone this discussion until Section \ref{linords}.

\section{Notation and once-off explanations}

The symbols $C$ and $c$ denote positive universal constants that may take on
different values at each appearance. $\mathbb{M}$ and $\mathbb{E}$ denote median and expected value. $n$ and $k$ will typically denote natural numbers and this will not always be stated explicitly but should be clear from the context. $\textrm{Lip}(f)\in \left[0,\infty\right]$ denotes the Lipschitz constant of any function $f:\mathbb{R}^n \rightarrow \mathbb{R}$ with respect to the Euclidean norm. $\sum_{i=a}^bf(i)$ denotes summation over all $i\in\mathbb{N}$ such that $a\leq i \leq b$, regardless of whether $a,b\in\mathbb{N}$. $1_{\{\cdot\}}$ denotes the indicator function of a set or condition. When proving a probability bound of the form $C\exp\left(-ct^2\right)$, we may take $C$ sufficiently large and assume in the proof that, say, $t\geq1$, because for $t<1$ the resulting probability bound is greater than $1$ and the result holds trivially. After such a bound is proved we may replace $C$ with $2$ using the fact that there exists $c'>0$ such that 
\[
\min\left\{1,C\exp\left(-ct^2\right)\right\}\leq 2\exp\left(-c't^2\right)
\]
Lastly, by making an all-round change of variables we may take $c'$ to be, say, $1$ or $1/2$. The constants in the final probability bound will therefore (without further explanation) not always match what appears to come from the proof.

\section{\label{Lo estimating}Lemmas}
We start with basic estimates for the lower incomplete gamma function suited to our purposes.
\begin{lemma}
\label{incomplete gamma}For all $b,q\in \left[ 0,\infty \right)$,
\begin{eqnarray*}
c^{1+q}\min\left\{1+q,b\right\}^{1+q}&\leq& \int_0^b e^{-\omega}\omega^qd\omega\leq C^{1+q}\min\left\{1+q,b\right\}^{1+q}\\
\frac{ce^bb^{1+q}}{1+q+b}&\leq& \int_0^b e^{\omega}\omega^qd\omega \leq \frac{Ce^bb^{1+q}}{1+q+b}
\end{eqnarray*}
\end{lemma}
\begin{proof}
The first integrand increases on $[0,q]$ and decreases on $[q,\infty)$, so for $b\leq q$, comparing the integral to the area of a large rectangle,
\[
\int_0^b e^{-\omega}\omega^qd\omega \leq be^{-b}b^q\leq b^{1+q}
\]
while for $b\geq1+q$,
\[
\int_0^b e^{-\omega}\omega^qd\omega \leq \Gamma(1+q)\leq C^{1+q}(1+q)^{1+q}\leq C^{1+q}b^{1+q}
\]
If $q\geq1$ this also holds for $q<b<1+q$, since in that case $(1+q)^{1+q}\leq C^{1+q}b^{1+q}$. So all that remains for the upper bound is the case where $0\leq q <1$ and $q<b<1+q$, which implies
\[
\int_0^b e^{-\omega}\omega^qd\omega \leq \int_0^b \omega^qd\omega \leq b^{1+q}
\]
We now consider the lower bound. For $b\leq q$, comparing the integral to the area of a smaller rectangle,
\[
\int_0^b e^{-\omega}\omega^qd\omega \geq \frac{b}{2}e^{-b/2}\left(\frac{b}{2}\right)^q\geq c^{1+q}b^{1+q}
\]
For $b\geq q$ and $q\geq 1$, using what we have just proved,
\[
\int_0^b e^{-\omega}\omega^qd\omega \geq \int_0^q e^{-\omega}\omega^qd\omega \geq c^{1+q}q^{1+q}\geq c^{1+q}(1+q)^{1+q}
\]
For $b\geq q$ and $0\leq q< 1$, we consider two sub-cases: firstly $b\leq 1+q$, in which case
\[
\int_0^b e^{-\omega}\omega^qd\omega \geq c\int_0^b \omega^qd\omega \geq c^{1+q}b^{1+q}
\]
and secondly $b> 1+q$, in which case
\[
\int_0^b e^{-\omega}\omega^qd\omega \geq c \int_0^1 \omega^qd\omega \geq c \geq c^{1+q}(1+q)^{1+q}
\]
The second integral with $e^\omega$ instead of $e^{-\omega}$ can be estimated by writing $e^{\omega}\omega^q=\exp\left( \omega+q\ln \omega\right)$ and using $\ln \omega \leq (\omega-b)/b+\ln b$, valid for all $\omega \in (0,b]$, and $\ln \omega \geq 2(\omega-b)/b+\ln b$, valid for all $\omega \in [b/2,b]$. Here we also use $ce^z/(1+z)\leq(e^z-1)/z\leq Ce^z/(1+z)$ valid for all $z>0$.
\end{proof}

We will use the fact that for any non-increasing function $f:[1,n]\rightarrow \mathbb{R}$,
\[
\frac{1}{2}\left(f(1)+\int_1^nf(x)dx\right)\leq\sum_{i=1}^nf(i)\leq f(1)+\int_1^nf(x)dx
\]

\begin{lemma}
\label{Lo basic sum bound}For all $a,q\in \left[ 0,\infty \right) $ and all $n\geq 2$ the following is
true: If $a\in \left[ 0,1\right] $ then
\[
\sum_{i=1}^{n }i^{-a}\left( \ln \frac{n}{i}\right) ^{q} \leq \frac{C^{1+q}n^{1-a}\left( 1+q\right) ^{1+q}\left( \ln n\right) ^{1+q}}{%
\left( \left( 1-a\right) \ln n+1+q\right) ^{1+q}}
\]%
and if $a\in %
\left[ 1,\infty \right) $ the sum is bounded above by
\[
\frac{C\left( \ln n\right) ^{1+q}}{\left( a-1\right) \ln n+1+q}+\left( \ln
n\right) ^{q}
\]
The corresponding lower bounds hold by replacing $C$ with $c$. When $q=0$ and $i=n$ in the sum, we consider $0^0=1$.
\end{lemma}

\begin{proof}
We focus on the upper bounds; the lower bounds follow the same steps. Integrals are estimated using Lemma \ref{incomplete gamma}, and we use the fact that $\min\{x,y\}$ is the same order of magnitude as $xy/(x+y)$. First, let $a\in \left( 0,\infty \right) $. Peeling off the first
term, comparing the remaining sum to an integral using monotonicity, and
setting $e^{s/a}=n/x$,%
\[
n^{-a}\sum_{i=1}^{n}\left( \frac{%
n}{i}\right) ^{a}\left( \ln \frac{n}{i}\right) ^{q}\leq \left( \ln n\right)
^{q}+\frac{n^{1-a}}{a^{1+q}}\int_0^{a\ln n}\exp \left( \left( 1-\frac{%
1}{a}\right) s\right) s^{q}ds 
\]%
If $a\in \left( 1,\infty \right) $ set $w=\left( 1-1/a\right) s$
to get%
\[
\left( \ln n\right) ^{q}+\frac{n^{1-a}}{\left( a-1\right) ^{1+q}}\int_0^{\left( a-1\right) \ln n}e^{w}w^{q}dw 
\]%
If $a=1$ we get $\left( \ln n\right) ^{q}+\int_{0}^{\ln n}s^{q}ds$ which can be absorbed into either the case $a\in \left( 0,1\right) $ or the case $a\in \left(
1,\infty \right) $. If $a\in \left( 0,1\right) $ then set $w=-\left( 1-1/a\right) s$ to get%
\[
\left( \ln n\right) ^{q}+\frac{n^{1-a}}{\left( 1-a\right) ^{1+q}}\int_0^{\left( 1-a\right) \ln n}e^{-w}w^{q}dw
\]%
If $a=0$, setting $w=\ln \left( n/x\right) $,%
\[
\sum_{i=1}^{n}i^{-a}\left( \ln 
\frac{n}{i}\right) ^{q}\leq \left( \ln n\right) ^{q}+\int_{1}^{n}\left(
\ln \frac{n}{x}\right) ^{q}dx\leq \left( \ln n\right) ^{q}+n\int_0^{\ln n}e^{-w}w^{q}dw 
\]
For $a\in \left[ 0,1\right]$ the factor $\left(\ln n \right)^{q}$ gets absorbed into the remaining term since
\[
\frac{C^{1+q}n^{1-a}\left( 1+q\right) ^{1+q}\left( \ln n\right) ^{1+q}}{%
\left( \left( 1-a\right) \ln n+1+q\right) ^{1+q}}=\frac{C^{1+q}n^{1-a}\left( \ln n\right) ^{1+q}}{%
\left(1+ \frac{\left( 1-a\right) \ln n}{1+q}\right) ^{1+q}}
\]
\end{proof}

The following lemma interpolates between the case $a=1$ and $a\neq 1$.
\begin{lemma}
\label{Lo basic power int}For all $\left(
a,T\right) \in \mathbb{R}\times \left[ 1,\infty \right) $,%
\[
c\frac{1+T^{1-a}}{1+\left\vert 1-a\right\vert \ln T}\ln T\leq
\int_{1}^{T}x^{-a}dx\leq C\frac{1+T^{1-a}}{1+\left\vert 1-a\right\vert \ln T}%
\ln T 
\]
\end{lemma}

\begin{proof}
First assume $a\neq 1$ and $T\neq 1$ and write%
\[
\int_{1}^{T}x^{-a}dx=\frac{\exp \left( \left( 1-a\right) \ln T\right) -1}{%
\left( 1-a\right) \ln T}\ln T 
\]%
Then interpret $s^{-1}\left( \exp \left( s\right) -1\right) $ as the slope
of a secant line and bound it above and below by $C\left( 1+e^{s}\right)
/\left( 1+\left\vert s\right\vert \right) $ in the cases $s\leq -1$, $s\in
\left( -1,1\right) \backslash \left\{ 0\right\} $ and $1\leq s$. Then notice
that the estimate also holds when $a=1$ and/or $T=1$.
\end{proof}

Define $\xi _{1}:\left[ 0,1\right] \rightarrow \left[
0,1\right]$ by $\xi _{1}(t)=e^{t}\left( 1-t\right)$, from which it follows, see \cite{Fr2}, that
\begin{eqnarray*}
\xi _{1}^{-1}(t) &\leq &\min \left\{ \sqrt{2\left( 1-t\right) }%
,1-e^{-1}t\right\} :0\leq t\leq 1
\end{eqnarray*}
The following lemma is taken from \cite{Fr2}, which is based on basic estimates for the binomial distribution and the R\'enyi representation of order statistics from the exponential distribution. Recall that the order statistics of a vector $x\in \mathbb{R}^n$ are denoted $\left(x_{(i)}\right)_1^n$ (the non-decreasing rearrangement of its coordinates), and the non-increasing rearrangement of the absolute values of the coordinates of $x$ are denoted $\left(x_{[i]}\right)_1^n$. So if all coordinates of $x$ are non-negative, then $x_{[i]}=x_{(n-i+1)}$.

\begin{lemma}
\label{concentration of order statistics}Let $\left( \gamma _{i}\right)
_{1}^{n}$ be an i.i.d. sample from $\left( 0,1\right) $ with corresponding
order statistics $\left( \gamma _{(i)}\right) _{1}^{n}$ and let $t>0$. With
probability at least $1-3^{-1}\pi ^{2}\exp \left( -t^{2}/2\right) $, the
following event occurs: for all $1\leq i\leq n$,
\begin{eqnarray}
\gamma _{(i)}\leq1-\frac{n-i+1}{n+1}\left( 1-\xi _{1}^{-1}\left( \exp \left( \frac{%
-t^{2}-4\ln \left( n-i+1\right) }{2(n-i+1)}\right) \right) \right)\label{bottom order}
\end{eqnarray}%
and with probability at least $1-C\exp \left( -t^{2}/2\right) $ the
following event occurs: for all $1\leq i\leq n$,%
\begin{eqnarray}
\gamma _{(i)}\leq1-\frac{n-i}{n}\exp \left( -c\max \left\{ \frac{\left(
t+\sqrt{\ln i}\right) \sqrt{i}}{\sqrt{n\left( n-i+1\right) }},\frac{%
t^{2}+\ln i}{n-i+1}\right\} \right)  \label{Renyi est}
\end{eqnarray}%
\end{lemma}

In Lemma \ref{concentration of order statistics}, (\ref{bottom order}) is preferable for $i>n/2$ while (\ref{Renyi est}) is preferable for $i\leq n/2$.

\begin{lemma}
\label{normalorderstats}Let $n\geq 3$, $t\geq 0$, and let $X$ and $Y$ be independent random vectors in $\mathbb{R}^n$, each with the standard normal distribution. Let $T:\mathbb{R}^n\rightarrow\mathbb{R}^n$ be the function that arranges the coordinates of a vector in non-decreasing order. Then with probability at least $1-C\exp\left(-t^2\right)$, the following event occurs: for all $1\leq i \leq (n+1)/2$,
\begin{equation}
X_{\left[ i\right] }\leq C\left( \ln \frac{n}{i}+\frac{t^{2}}{i}\right)
^{1/2} \label{orderstat estimate}
\end{equation}
and
\begin{equation}
\left\vert TX-TY\right\vert_{\infty}\leq C\min\left\{\frac{t^2}{\sqrt{\ln n}},t\right\}\label{ellinftyest}
\end{equation}
and with probability at least $0.51$ the following event occurs: for all $1\leq i \leq (n+1)/2$,
\begin{equation}
c\sqrt{\ln \frac{n}{i}}\leq X_{[i]}\leq C\sqrt{\ln \frac{n}{i}}\label{simultaneous medians}
\end{equation}
\end{lemma}

\begin{proof}
We will not repeat `with probability...' as it is clear that there are probabilities associated to the events in question, and what those probabilities are. To prove estimates for a general distribution based on estimates for the uniform distribution, transform under the action of the inverse cumulative distribution, which is increasing and preserves the operation of arranging in non-increasing order. So we may write $X_{[i]} =\Phi ^{-1}\left( \frac{1+\gamma _{(n-i+1)}}{2}\right) $, where $\left( \gamma _{i}\right) _{1}^{n}$ is an i.i.d. sample
from the uniform distribution on $\left( 0,1\right)$, and we use the bound
\begin{equation}
c\sqrt{\ln \frac{1}{1-x}}\leq \Phi ^{-1}\left( \frac{1+x}{2}\right)\leq C\sqrt{\ln \frac{1}{1-x}}\label{inverse n bound}
\end{equation}
valid for all $x\in(1/3,1)$, say. (\ref{orderstat estimate}) now follows from (\ref{bottom order}) and (\ref{inverse n bound}), and includes the upper bound in (\ref{simultaneous medians}) as a special case. To get a lower bound on an order statistic, we apply Lemma \ref{concentration of order statistics} to the i.i.d. random variables $\left( 1-\gamma _{i}\right)_{1}^{n}$, which are also uniformly distributed in $(0,1)$ and whose vector of order statistics is $\left( 1-\gamma _{(n-i+1)}\right)_{1}^{n}$. An upper bound on $1-\gamma _{(n-i+1)}$ translates to a lower bound on $\gamma _{(n-i+1)}$. So, from (\ref{Renyi est}) and (\ref{inverse n bound}),
\[
X_{\left[ i\right] }\geq c\sqrt{\ln \left( 1-\frac{n-i}{n}\exp \left( -\frac{%
C\left(1+\sqrt{\ln i}\right)\sqrt{i}}{n}\right) \right) ^{-1}}\geq c\sqrt{\ln 
\frac{n}{i}} 
\]%
which is seen to hold when $n>n_{0}$ (using $e^{-z}\geq 1-z$), and when $n\leq n_{0}$ this is bounded below by $c_{2}\geq c\sqrt{%
\ln \left( n/i\right) }$.

We now consider (\ref{ellinftyest}). Its proof, which occupies the next two pages, may later be removed and placed in another paper. It follows from the bounds relating $\Phi $ and $\phi=\Phi'$ that
\begin{equation}
\frac{d}{dx}\Phi ^{-1}\left( x\right) =\frac{1}{\phi \left( \Phi ^{-1}\left(
x\right) \right) }\leq \frac{C}{\min \left\{ x,1-x\right\} \sqrt{\ln \min
\left\{ x,1-x\right\} ^{-1}}}\label{derivative}
\end{equation}
A difference between our current calculations for (\ref{ellinftyest}) and what has been done for (\ref{orderstat estimate}) and (\ref{simultaneous medians}) above, is that there are no absolute values involved in the definition of $X_{(i)}$, and we write $X_{(i)} =\Phi ^{-1}\left(\gamma _{(i)}\right) $, where $\left( \gamma _{i}\right) _{1}^{n}$ is an i.i.d. uniform sample. Here we are re-using notation, and this $\left( \gamma _{i}\right) _{1}^{n}$ is not the same as the previous $\left( \gamma _{i}\right) _{1}^{n}$, which is inconsequential since we are now doing a new calculation. So, by (\ref{derivative}) and (\ref{bottom order}), for all $n/2<i\leq n$, $\Phi ^{-1}\left( \gamma _{(i)}\right) -\Phi ^{-1}\left( \frac{i}{n+1}\right) $ is bounded above by
\begin{eqnarray*}
&&C\int_{i/(n+1)}^{1-\frac{n-i+1}{n+1}\left( 1-\xi _{1}^{-1}\exp \left( 
\frac{-t^{2}-4\ln (n-i+1)}{2\left( n-i+1\right) }\right) \right) }\left(
1-x\right) ^{-1}\left( \ln \left( 1-x\right) ^{-1}\right) ^{-1/2}dx \\
&=&C\int_{\ln \frac{n+1}{n-i+1}}^{\ln \frac{n+1}{n-i+1}-\ln \left( 1-\xi
_{1}^{-1}\exp \left( \frac{-t^{2}-4\ln (n-i+1)}{2\left( n-i+1\right) }%
\right) \right) }s^{-1/2}ds \\
&\leq &C\left( \ln \frac{n+1}{n-i+1}\right) ^{-1/2}\ln \left( 1-\xi _{1}^{-1}\exp \left( \frac{-t^{2}-4\ln
(n-i+1)}{2\left( n-i+1\right) }\right) \right) ^{-1}
\end{eqnarray*}%
The function $s\mapsto -\ln \left( 1-\xi _{1}^{-1}\exp \left( -s\right)
\right) $ behaves like $\sqrt{s}$ near $0$ and like $s+1$ when $s$ is
large. We then consider two cases, depending on whether%
\begin{eqnarray}
\frac{t^{2}+4\ln (n-i+1)}{2\left( n-i+1\right) }
\end{eqnarray}%
lies in $\left( 0,1\right) $ or $\left[ 1,\infty \right) $, and in either
case the estimate is bounded above by $Ct^{2}/\sqrt{\ln n}$. Here we have used the fact that for all $a,b\geq 1$,
\[
\ln \left(1+ab\right)\leq C\ln \left(1+a\right)\ln \left(1+b\right)
\]
which is true since $1+ab\leq(1+a)(1+b)$ and for positive numbers uniformly bounded away from $0$ a product dominates a sum up to a constant. Therefore
\[
\ln \left(1+\frac{n}{n-i+1}\right)\ln \left(1+n-i+1\right)\geq c\ln n
\]
which can be modified be deleting the leftmost $1+$ and changing the factor $\ln (1+n-i+1)$ to something larger such as $n-i+1$ or $(n-i+1)/\ln (n-i+1)$. A zero in the denominator doesn't hurt since we are in practice considering the reciprocals. And assuming as we may that $t\geq 1$, $1+t\leq Ct^2$. So this is where the upper bound $Ct^2/\sqrt{\ln n}$ comes from. Obviously this bound can be improved significantly for individual order statistics; we haven't bothered to write out such bounds since for our purposes we need a uniform estimate over all $i$. A similar
calculation with a lower bound for $\gamma _{(i)}$ follows from (\ref{Renyi est}) applied to $\left(
1-\gamma _{(i)}\right) _{1}^{n}$: $\Phi ^{-1}\left( \frac{i-1}{n}\right)-\Phi ^{-1}\left( \gamma _{(i)}\right)$ is bounded above by
\begin{eqnarray*}
&&C\int_
{\frac{i-1}{n}\exp\left(-c\max\left\{\frac{\left(t+\sqrt{\ln (n-i+1)}\right)\sqrt{n-i+1}}{n},\frac{t^2}{n}\right\}\right)}^{(i-1)/n}\left(
1-x\right) ^{-1}\left( \ln \left( 1-x\right) ^{-1}\right) ^{-1/2}dx \\
&=&C\int_{-\ln \left[1-\frac{i-1}{n}\exp\left(-c\max\left\{\frac{\left(t+\sqrt{\ln (n-i+1)}\right)\sqrt{n-i+1}}{n},\frac{t^2}{n}\right\}\right)\right]}^{\ln \frac{n}{n-i+1}}s^{-1/2}ds 
\end{eqnarray*}
Using $\int_a^bs^{-1/2}ds=2(b-a)/(\sqrt{b}+\sqrt{a})\leq C(b-a)/\sqrt{b}$ valid for $0<a<b$, applied to this last quantity,
\[
\sqrt{b}=\left(\ln \frac{n}{n-i+1}\right)^{1/2}
\]
and $b-a$ is equal to
\begin{eqnarray*}
&C&\ln \left[1+\frac{i-1}{n-i+1}\left(1-\exp\left(-c\max\left\{\frac{\left(t+\sqrt{\ln (n-i+1)}\right)\sqrt{n-i+1}}{n},\frac{t^2}{n}\right\}\right)\right)\right]\\
&\leq&C\frac{i-1}{n-i+1}\left(1-\exp\left(-c\max\left\{\frac{\left(t+\sqrt{\ln (n-i+1)}\right)\sqrt{n-i+1}}{n},\frac{t^2}{n}\right\}\right)\right)
\end{eqnarray*}
The expression $1-e^{-s}$ behaves like $s$ for $0\leq s\leq 1$ and like $1$ when $s>1$. We now consider two cases depending on whether
\[
c\max\left\{\frac{\left(t+\sqrt{\ln (n-i+1)}\right)\sqrt{n-i+1}}{n},\frac{t^2}{n}\right\}
\]
lies in $[0,1]$ or $(1,\infty)$. In the first case, for the entire expression $C(b-a)/\sqrt{b}$, we get the same bound $Ct^2/\sqrt{\ln n}$ as before, using similar simplifications. In the second case we get
\[
\frac{Cn}{(n-i+1)\sqrt{\ln \frac{n}{n-i+1}}}
\]
However for the expression defining this case to be $> 1$,
\[
t\geq c\min\left\{\sqrt{n},\frac{n}{\sqrt{\ln (n-i+1)}}\right\}
\]
and regardless of which term defines this minimum we end up with the same bound $Ct^2/\sqrt{\ln n}$. We must now handle the discrepancy between $i/(n+1)$ and $(i-1)/n$ in the computations above. From (\ref{derivative}),
\[
\left\vert\Phi ^{-1}\left( \frac{i}{n+1}\right)-\Phi ^{-1}\left( \frac{i-1}{n}\right)\right\vert\leq
\frac{C}{n}\left(1-\frac{i}{n+1}\right)^{-1}\left(\ln \left(1-\frac{i}{n+1}\right)^{-1}\right)^{-1/2}
\leq \frac{C}{\sqrt{\ln n}}
\]
All of this implies that
\[
\left\vert X_{(i)}-\Phi^{-1} \left(\frac{i}{n+1}\right) \right\vert \leq \frac{Ct^2}{\sqrt{\ln n}}
\]
Similar bounds in the case $1\leq i\leq n/2$ now follow by
symmetry, and they also hold for $Y$ since $Y$ has the same distribution as $X$, and a bound on $\left\vert TX-TY\right\vert_{\infty}$ follows by the triangle inequality. For large values of $t$ this can be improved as follows. Let $x^*\in \mathbb{R}^n$ be defined as
\[
x^*_i=\Phi^{-1} \left(\frac{i}{n+1}\right)
\]
Since $T$ acts as an isometry on each of the $n!$ overlapping regions of $\mathbb{R}^n$ determined by the order of the coordinates of a vector, $T$ is $1$-Lipschitz on $\mathbb{R}^n$. So $x\mapsto Tx \mapsto Tx-x^*\mapsto \left\vert Tx-x^* \right\vert_\infty$ is the composition of $1$-Lipschitz functions and by Gaussian concentration,
\[
\left\vert TX-x^* \right\vert_\infty\leq \mathbb{M}\left\vert TX-x^* \right\vert_\infty+Ct
\]
which also applies to $TY$, and the result follows again by the triangle inequality.
\end{proof}

\begin{lemma}
\label{Lo med calculation}Let $0\leq r<\infty$, $1\leq p<\infty$, $n\geq2$, and let $X$ be a random vector in $\mathbb{R%
}^{n}$ with the standard normal distribution. If $r\in \left[ 0,1\right]$ then
\[
\mathbb{M}\sum_{i=1}^{n}i^{-r}X_{\left[ i\right] }^{p}\leq \frac{C^{p}p^{p/2}n^{1-r}\left( \ln n\right) ^{1+p/2}}{\left[p+(1-r) \ln n\right] ^{1+p/2}}
\]
and if $r\in \left[ 1,\infty \right)$ then
\[
\mathbb{M}\sum_{i=1}^{n}i^{-r}X_{\left[ i\right] }^{p}\leq \frac{C^p\left( \ln n\right)^{1+p/2}}{1+(r-1) \ln n}+C^p\left(\ln n\right)^{p/2} 
\]
with the reverse inequalities holding with $C$ replaced by $c$.
\end{lemma}

\begin{proof}
The result follows from Eq. (\ref{simultaneous medians}) of Lemma \ref{normalorderstats}, together with Lemma \ref{Lo basic
sum bound}.
\end{proof}

\begin{lemma}
\label{Lo Lip con}Let $0\leq r<\infty$ and $1\leq p<\infty$. Then%
\[
\sup \left\{ \left( \sum_{i=1}^{n}i^{-r}\theta _{\left[ i\right]
}^{p}\right) ^{1/p}:\theta \in S^{n-1}\right\} =\left\{ 
\begin{array}{ccc}
\left( \sum_{1}^{n}i^{-2r/(2-p)}\right) ^{(2-p)/2p} & : & p\in \left[
1,2\right) \\ 
1 & : & p\in \left[ 2,\infty \right)%
\end{array}%
\right. 
\]%
For $p\in \left[ 1,2\right) $ this can be bounded above by%
\[
1+C\left( \frac{\ln n}{1+\left\vert2-2r-p\right\vert
\ln n}\right) ^{\frac{2-p}{2p}}\left( 1+n^{\frac{2-2r-p}{2p}}\right) 
\]
and below by the same quantity with $C$ replaced with $c$. For $r<r_0$ (for any universal constant $r_0>1$) the leftmost `1+' can be deleted.
\end{lemma}

\begin{proof}
For $p\in \left[ 1,2\right) $ an upper bound follows by H\"{o}lder's
inequality for $\ell _{2/(2-p)}^{n}-\ell _{2/p}^{n}$ duality, with equality
when $\theta _{i}=i^{-r/(2-p)}\left( \sum_{j=1}^{n}j^{-2r/(2-p)}\right)
^{-1/2}$. Now 
\[
\left( \sum_{1}^{n}i^{-2r/(2-p)}\right) ^{(2-p)/2p}\leq \left(
1+\int_{1}^{n}x^{-2r/(2-p)}dx\right) ^{(2-p)/2p} 
\]%
which is bounded using Lemma \ref{Lo basic power int} and noting that $%
0<(2-p)/(2p)\leq 1/2$ and that $c<(2-p)^{(2-p)}<C$. For $p\in \left[ 2,\infty \right) $, $\left(
\sum_{i=1}^{n}i^{-r}\theta _{\left[ i\right] }^{p}\right) ^{1/p}\leq \left(
\sum_{i=1}^{n}\theta _{\left[ i\right] }^{p}\right) ^{1/p}\leq 1$ with
equality when $\theta =e_{1}$.
\end{proof}

We shall use the fact that for all $b\in [1,n]$, not necessarily an integer,
\[
\sum_{i=1}^{n}i^{-r}x _{\left[ i\right]}^{p}\leq \frac{2n}{b}\sum_{i=1}^{b}i^{-r}x _{\left[ i\right]}^{p}
\]
where the sum on the right is over all $i\in \mathbb{N}$ such that $1\leq i \leq b$.

\begin{lemma}
\label{diff restrict}Let $\psi:\mathbb{R}^n\mapsto\mathbb{R}$ and let $A\subset\mathbb{R}^n$ be convex set with nonempty interior. Let $E$ denote the collection of all $x\in\mathbb{R}^n$ such that the coordinates of $x$ are distinct and non-zero. Assume that $\psi$ is continuous on $A$ and differentiable on $A \cap E$. Then, as an element of $[0,\infty]$,
\[
\mathrm{Lip}\left(\psi|_A\right)=\sup_{x\in A\cap E}\left\vert\nabla\psi (x)\right\vert
\]
\end{lemma}
\begin{proof}
(Sketch) We focus on proving that $LHS\leq RHS$ when $RHS<\infty$; the rest comes down to approximating a supremum. Note that $E$ is dense in $\mathbb{R}^n$, and by considering sets of the form $\mathrm{conv}\{z,B\}$ where $z\in A$ and $B$ is any ball of non-zero radius contained in $A$, we see that $\mathrm{int}\left(A\right)$ is dense in $A$. So, suppose $\left\vert\nabla\psi (z)\right\vert\leq L$ for all $z\in A\cap E$ and some $L<\infty$, and consider any $x_1,y_1\in A$ with $x_1\neq y_1$, and any $\varepsilon>0$. We consider sequences $x_2,x_3,x_4$ and $y_2,y_3,y_4$ such that each $x_{i+1}$ is sufficiently close to its predecessor $x_i$ and by continuity each $\psi\left(x_{i+1}\right)$ is sufficiently close to $\psi\left(x_i\right)$. Similarly so with the $y$'s. This can be done so that $x_2,y_2\in \mathrm{int}(A)$, $x_3,y_3\in \mathrm{int}(A)$ and that all coordinates of $x_3$ and $y_3$ are distinct ($2n$ distinct coordinates in total), and that $x_4,y_4\in A\cap E$ such that $x_4$ and $y_4$ also have completely distinct coordinates between the two of them. We bound $\left\vert \psi(x_1)-\psi(y_1)\right\vert$ up to a term involving $\varepsilon$ by integrating over the line segment joining $x_4$ and $y_4$, which contains at most $n+n(n-1)/2$ points not in $E$. Details are left to the reader.
\end{proof}

\section{\label{linords}Order statistics, norms and quasi-norms}

As part of the proof of Theorem \ref{LoMain}, towards estimating the distribution of a gradient in order to derive concentration inequalities, we need estimates for the distribution of certain functionals of order statistics. However, in order to avoid the blowup of a quasi-norm constant as $p\rightarrow 1$ we need more. We need:
\begin{itemize}
\item A deterministic bound of the form
\[
\sum_{i=1}^{n}i^{-2r}x_{[i]}^{2(p-1)} \leq \varphi \left( \left \Vert x \right \Vert \right) 
\]
where $\left \Vert \cdot \right \Vert$ is a norm, $\varphi:\left[0,\infty \right) \rightarrow \left[0,\infty \right)$, and the inequality is valid for all $x\in \mathbb{R}^n$,
\item a bound on the distribution of $\varphi \left( \left \Vert X \right \Vert \right)$, where $X$ is a random vector in $\mathbb{R}^n$ with the standard normal distribution.
\end{itemize}
Since these estimates will affect the bounds we end up with in Theorem \ref{LoMain}, the problem is not purely existential; we want good estimates. Of greatest interest is the case $p\in \left(1,3/2\right)$ where the quantity $\sum_{i=1}^{n}i^{-2r}x_{[i]}^{2(p-1)}$ is not already a function of a norm of $x$. Geometrically, the problem is related to finding a convex subset of a given non-convex set whose complement has comparable Gaussian measure to the complement of the given non-convex set.

\bigskip

\begin{tikzpicture}
\draw[white] (0,0) circle(0pt);
\draw (2.85,0) .. controls (5.85,1) .. (6.85,4);
\draw (6.85,4) .. controls (7.85,1) .. (10.85,0);
\draw (10.85,0) .. controls (7.85,-1) .. (6.85,-4);
\draw (6.85,-4) .. controls (5.85,-1) .. (2.85,0);
\draw (4.35,0) -- (6.85,2.5);
\draw (6.85,2.5) -- (9.35,0);
\draw (9.35,0) -- (6.85,-2.5);
\draw (6.85,-2.5) -- (4.35,0);
\draw[black] (6,-5) circle(0pt) node[anchor=west]{Figure 2};
\end{tikzpicture}

\begin{theorem}\label{orderorderbound}
Let $X$ be a random vector in $\mathbb{R}^n$, $n\geq3$, with the standard normal distribution, $0\leq r<\infty$, $1\leq p<\infty$ and $t>0$. In each of the following four cases, definitions are given for $R$, $S$ are $\left\vert \cdot \right \vert_{\sharp}$, and in each case,
\[
\mathbb{P}\left\{\left\vert X \right \vert_{\sharp}\leq S\right\}\geq 1-C\exp\left(-t^2/2\right) \hspace{1cm}\textit{and} \hspace{1cm} \left\vert X \right \vert_{\sharp}\leq S \Rightarrow \sum_{i=1}^{n}i^{-2r}X_{[i]}^{2(p-1)}\leq R
\]

\textbf{Case I:} If $p\in \left[3/2,\infty \right)$ then for all $x\in \mathbb{R}^n$ set
\[
 \left\vert x \right \vert_{\sharp}=\left(\sum_{i=1}^{n}i^{-2r}x_{[i]}^{2(p-1)}\right)^{\frac{1}{2(p-1)}}
 \]
 and $R=S^{2(p-1)}=A+Bt^{2(p-1)}$ where
 \[
 A=\frac{C^p p^p n^{1-2r} \left( \ln n \right)^p}{\left[ p+(1-2r) \ln n\right]^p}1_{\{0\leq r\leq 1/2\}}+
 \left(\frac{C^p\left( \ln n \right)^p}{1+(2r-1) \ln n }+C^p\left(\ln n \right)^{p-1}\right)1_{\{1/2< r<\infty\}}
 \]
 and
 \[
 B=C\left[1+\left( \frac{\ln n}{1+\left\vert 2-2r-p \right\vert \ln n} \right)^{2-p}\left( 1+n^{2-2r-p} \right)\right]1_{\{3/2\leq p< 2\}}+C^p1_{\{2\leq p< \infty\}}
 \]
In this case $\left\vert \cdot \right \vert_{\sharp}$ is a norm.

\textbf{Case II:} If $p\in \left[1,3/2 \right)$ then for all $x\in \mathbb{R}^n$,
\[
\sum_{i=1}^{n}i^{-2r}x_{[i]}^{2(p-1)} \leq C\left( \sum_{i=1}^{n/e}i^{-2r}\left( \ln \frac{n}{i}+\frac{t^{2}
}{i}\right) ^{p-1}\right) ^{3-2p}\left\vert x \right \vert_{\sharp}^{2\left(
p-1\right) }
\]
where
\begin{equation}
\left\vert x \right \vert_{\sharp}=\sum_{i=1}^{n/e}\frac{i^{-2r}x_{\left[ i\right] }}{\left( \ln
\frac{n}{i} +\frac{t^{2}}{i}\right) ^{\frac{3-2p}{2}}} \label{nomnom}
\end{equation}
and
\begin{eqnarray*}
R&=&S=C\sum_{i=1}^{n/e}i^{-2r}\left( \ln \frac{n}{i}+\frac{t^{2}
}{i}\right) ^{p-1}\\
&\leq&\frac{Cn^{1-2r}\left(\ln n\right)^p}{\left[1+(1-2r)\ln n \right]^p}1_{\{0\leq r \leq 1/2\}}+
C\left(\frac{\left(\ln n \right)^p}{1+(2r-1)\ln n}+\left(\ln n \right)^{p-1} \right)1_{\{1/2< r<\infty\}}\\
&+&C\left(1+\frac{1+n^{2-2r-p}}{1+\left\vert 2-2r-p\right\vert\ln n}\ln n \right)t^{2(p-1)}
\end{eqnarray*}
Under the added assumption that $p \geq 3/2-2r$, and because the sums have been restricted to $1\leq i \leq n/e$, $\left\vert \cdot \right \vert_{\sharp}$ is a norm.

\textbf{Case III:} If $p<3/2-2r$ (so necessarily $0\leq r <1/4$ and $1\leq p <3/2$), then for all $x\in \mathbb{R}^n$,
\[
\sum_{i=1}^{n}i^{-2r}x_{[i]}^{2(p-1)} \leq Cn^{(1-2r)(3-2p)}\left\vert x \right \vert_{\sharp}^{2(p-1)}
\]
where
\[
\left\vert x \right \vert_{\sharp}=\sum_{i=1}^{n}i^{-2r}x_{[i]}
\]
and
\begin{eqnarray*}
S&=&Cn^{1-2r}+Cn^{\frac{1-4r}{2}}\left( \frac{\ln n}{1+(1-4r) \ln n} \right)^{\frac{1}{2}}t \\
R&=&Cn^{(1-2r)(3-2p)}S^{2(p-1)}\leq Cn^{1-2r}+Cn^{2-2r-p}\left(\frac{\ln n}{1+(1-4r)\ln n}\right)^{p-1}t^{2(p-1)}
\end{eqnarray*}

\textbf{Case IV:} When $r\in \left(1/4,1/2\right]$ and $p=2(1-r)$, in which case $p\in \left[1,3/2 \right)$, the result in Case II can be improved as follows:

\underline{Case IVa:} If $(1-2r)\ln n \geq e$ (which excludes the case $p=1$), then for all $x\in \mathbb{R}^n$,
\[
\sum_{i=1}^{n}i^{-2r}x_{[i]}^{2(p-1)} \leq C\left(\ln n\right)^{3-2p} \left\vert x \right \vert_{\sharp}^{2(p-1)}
\]
where
\[
\left\vert x \right \vert_{\sharp}=\sum_{i=1}^{n/e} \beta_i^{\frac{-(3-2p)}{2(p-1)}}i^{\frac{-r}{p-1}}x_{[i]},
\]
and
\[
\beta_i=\frac{(1-2r)^p \ln n}{n^{1-2r}}i^{-2r}\left( \ln \frac {n}{i}\right)^{p-1}+i^{-1}
\]
The coefficient $\beta_i^{\frac{-(3-2p)}{2(p-1)}}i^{\frac{-r}{p-1}}$ is non-increasing in $i$ for $1\leq i \leq n/e$ (so that $\left\vert \cdot \right \vert_{\sharp}$ is a norm). In this case
\[
S=C^{\frac{1}{p-1}}(1-2r)^{\frac{-p}{2(p-1)}}\left( \ln n \right)^{\frac{-(3-2p)}{2(p-1)}}n^{\frac{1}{2}}+C^{\frac{1}{p-1}}\left( \ln n \right)^{\frac{1}{2}}t
\]
and
\[
R=C\left(\ln n\right)^{3-2p} S^{2(p-1)}\leq C(1-2r)^{-1}n^{1-2r}+C\left( \ln n \right)^{2-p}t^{2(p-1)}
\]

\underline{Case IVb:} If $(1-2r)\ln n < e$, then for all $x\in \mathbb{R}^n$,
\[
\sum_{i=1}^{n}i^{-2r}x_{[i]}^{2(p-1)} \leq C\left( \ln n\right) \left \vert x\right \vert_{\sharp}^{2(p-1)}
\]
where $\left \vert \cdot \right \vert_{\sharp}=\left\vert\cdot\right\vert$ is the standard Euclidean norm, $S=Cn^{\frac{1}{2}}+t$, and
\[
R=C\left( \ln n\right) S^{2(p-1)}\leq C\left( \ln n \right)t^{2(p-1)}
\]
\end{theorem}
\begin{proof}
\textbf{Case I:} $p\in \left[3/2, \infty \right)$. In this case $\left \vert \cdot \right \vert_{2r,2(p-1)}$ is a norm and it follows by classical Gaussian concentration that with probability at least $1-C\exp\left(-t^2/2\right)$,
\[
\sum_{i=1}^{n}i^{-2r}X_{[i]}^{2(p-1)}=\left \vert X \right \vert_{2r,2(p-1)}^{2(p-1)}\leq \left[ \mathbb{M}\left \vert X \right \vert_{2r,2(p-1)}+t \mathrm{Lip}\left \vert \cdot \right \vert_{2r,2(p-1)} \right]^{2(p-1)}
\]
Estimates for $\mathbb{M}\left \vert X \right \vert_{2r,2(p-1)}$ and $\mathrm{Lip}\left \vert \cdot \right \vert_{2r,2(p-1)}$ follow from Lemmas \ref{Lo med calculation} and \ref{Lo Lip con}, and we leave the computation to the reader (with the reminder that we are applying these results with $2r$ and $2(p-1)$ instead of $r$ and $p$). Certain numerical simplifications can be made based on the values of $p$ and $r$ and the existance of the factor $C^p$.

\noindent \textbf{Proof. Case II:} We consider $p\in (1,3/2)$ and reclaim the case $p=1$ by taking a limit. For all $x\in \mathbb{R}^n$, by H\"{o}lder's inequality, $\sum_{1}^{n}i^{-2r}x_{[i]}^{2(p-1)}$ is bounded above by
\begin{eqnarray*}
&&C\sum_{i=1}^{n/e}\frac{i^{-4r\left( p-1\right) }x_{\left[ i\right]
}^{2(p-1)}}{\left( \ln \left( n/i\right) +t^{2}/i\right) ^{\left( p-1\right)
\left( 3-2p\right) }}i^{-2r\left( 3-2p\right) }\left( \ln \left( n/i\right)
+t^{2}/i\right) ^{\left( p-1\right) \left( 3-2p\right) } \\
&\leq&C\left( \sum_{i=1}^{n/e}\frac{i^{-2r}x_{\left[ i\right] }}{\left( \ln
\left( n/i\right) +t^{2}/i\right) ^{\left( 3-2p\right) /2}}\right) ^{2\left(
p-1\right) }\left( \sum_{i=1}^{n/e}i^{-2r}\left( \ln \frac{n}{i}+\frac{t^{2}%
}{i}\right) ^{p-1}\right) ^{3-2p}\\
&= &C\left\vert x \right\vert_{\sharp}^{2(p-1)}\left( \sum_{i=1}^{n/e}i^{-2r}\left( \ln \frac{n}{i}+\frac{t^{2}%
}{i}\right) ^{p-1}\right) ^{3-2p}
\end{eqnarray*}%
Eq. (\ref{orderstat estimate}) from Lemma \ref{normalorderstats} then implies that
\[
\left\vert X \right\vert_{\sharp}=\sum_{i=1}^{n/e}\frac{i^{-2r}X_{\left[ i\right] }}{\left( \ln\left( n/i\right) +t^{2}/i\right) ^{\left( 3-2p\right) /2}}
\leq C'\sum_{i=1}^{n/e}i^{-2r}\left( \ln \frac{n}{i}+\frac{t^{2}}{i}\right) ^{p-1}=\frac{S}{C''}
\]
Assuming this event occurs, the above calculation involving H\"{o}lder's inequality implies that $\sum_{1}^{n}i^{-2r}X_{[i]}^{2(p-1)}$ is bounded above by $S$. Lastly,
\begin{eqnarray*}
S\leq C\sum_{i=1}^{n/e}i^{-2r}\left( \ln \frac{n}{i}\right)
^{p-1}+Ct^{2(p-1)}\sum_{i=1}^{n/e}i^{-2r-p+1}
\end{eqnarray*}
which is bounded above using Lemmas \ref{Lo basic sum bound} and \ref{Lo basic power int}.

\noindent \textbf{Proof. Case III:} $p<3/2-2r$ (so necessarily $0\leq r <1/4$ and $1\leq p <3/2$). For $p\neq 1$, by H\"{o}lder's inequality,
\[
\sum_{i=1}^{n}i^{-2r}x_{[i]}^{2(p-1)}\leq \left(\sum_{i=1}^{n}i^{-2r}x_{[i]}\right)^{2(p-1)}\left(\sum_{i=1}^{n}i^{-2r}\right)^{3-2p}\leq Cn^{(1-2r)(3-2p)}\left\vert x \right\vert_{\sharp}^{2(p-1)}
\]
and for $p=1$ the same bound is seen to hold. An upper bound on the quantiles of $\left\vert X \right\vert_{\sharp}$ follows from Gaussian concentration (making use of Lemmas \ref{Lo basic power int} and \ref{Lo med calculation}).

\noindent \textbf{Proof. Case IV:} $p=2(1-r)$ and $r\in \left(1/4,1/2\right]$, in which case $p\in \left[1,3/2\right)$. The sub-case $(1-2r)\ln n < e$ is clear enough by H\"{o}lder's inequality for $p\neq 1$,
\begin{eqnarray*}
\sum_{i=1}^{n}i^{-2r}x_{[i]}^{2(p-1)} &\leq& \left( \sum_{i=1}^{n}\left(i^{-2r}\right)^{\frac{1}{2-p}}\right)^{2-p}\left( \sum_{i=1}^{n}\left(x_{[i]}^{2(p-1)}\right)^{\frac{1}{p-1}}\right)^{p-1}
\end{eqnarray*}
and then noting that the exponent $-2r/(2-p)=-1$ and applying classical Gaussian concentration to $\left\vert \cdot \right\vert$. For $p=1$ there is nothing to show. In this sub-case,
\[
\frac{\left(\ln n \right)^{2-p}}{\ln n}=\exp \left( -(1-2r)\ln\ln n \right)\in \left[e^{-1},1\right] \hspace{1.5cm} n^{1-2r}=\exp\left((1-2r)\ln n\right)\in \left[1,e^e\right)
\]
which is how we simplify the exponents of $\ln n$.

The rest of the proof deals with the other sub-case $(1-2r)\ln n \geq e$. Throughout, we make use of the relation $p=2(1-r)$ which is not always explicitly re-stated, and the reader should make a mental note of this. The case $p=1$ is automatically excluded from this sub-case. By taking the constant $C$ in the probability bound to be at least $\sqrt{e}$ we may assume that $t\geq1$. For any sequence $\left(\alpha_i\right)_1^{\left \lfloor n/e\right \rfloor}$ with $\alpha_i>0$, by H\"{o}lder's inequality,
\begin{eqnarray*}
\sum_{i=1}^{n}i^{-2r}x_{[i]}^{2(p-1)} &\leq&C\left(\sum_{i=1}^{n/e}\alpha_i^{\frac{1}{3-2p}}\right)^{3-2p}\left( \sum_{i=1}^{n/e}i^{\frac{-r}{p-1}}\alpha_i^{\frac{-1}{2(p-1)}}x_{[i]} \right)^{2(p-1)}\\
&=& C\left(\sum_{i=1}^{n/e}\beta_{i}\right)^{3-2p}\left( \sum_{i=1}^{n/e}i^{\frac{-r}{p-1}}\beta_{i}^{\frac{-(3-2p)}{2(p-1)}}x_{[i]} \right)^{2(p-1)}
\end{eqnarray*}
where $\beta_i=\alpha_i^{\frac{1}{3-2p}}$. This vector $\beta\in\mathbb{R}^{\left \lfloor n/e\right \rfloor}$ is considered a variable for now, and its value will later be fixed to match the value quoted in the statement of the result. Summing only up to $n/e$ will ensure that $i^{\frac{2r}{3-2p}}\beta_i$ is non-decreasing in $i$, which then implies that
\begin{equation}
\left\vert x\right\vert_{\sharp}=\sum_{i=1}^{n/e}i^{\frac{-r}{p-1}}\beta_{i}^{\frac{-(3-2p)}{2(p-1)}}x_{[i]} \label{normee}
\end{equation}
is a norm. By classical Gaussian concentration applied to $\left\vert \cdot \right\vert_{\sharp}$, with probability at least $1-C\exp\left(-t^2/2\right)$,
\begin{equation}
\sum_{i=1}^{n}i^{-2r}X_{[i]}^{2(p-1)}\leq \left(\sum_{i=1}^{n/e}\beta_{i}\right)^{3-2p} \left[\mathbb{M} \left\vert X \right\vert_{\sharp} +t \mathrm{Lip}\left( \left\vert \cdot \right\vert_{\sharp}\right) \right]^{2(p-1)} \label{essential bound}
\end{equation}
The median can be estimated using (\ref{orderstat estimate}) and the Lipschitz constant computed as the Euclidean norm of the gradient, which gives
\begin{eqnarray*}
\mathbb{M} \left\vert X \right\vert_{\sharp} &\leq& C\sum_{i=1}^{n/e} i^{\frac{-r}{p-1}} \beta_i^{\frac{-(3-2p)}{2(p-1)}}\left(\ln \frac{n}{i}\right)^{1/2}\\
\mathrm{Lip}\left( \left\vert \cdot \right\vert_{\sharp}\right) &=&\left(\sum_{i=1}^{n/e} i^{\frac{-2r}{p-1}} \beta_i^{\frac{-(3-2p)}{p-1}} \right)^{1/2}
\end{eqnarray*}
We temporarily assume that $\sum\beta_{i}=1$, which we may do by homogeneity, although this condition will later be relaxed. We wish to minimize the function
\[
\psi\left( \beta \right)=\sum_{i=1}^{n/e} i^{\frac{-r}{p-1}} \beta_i^{\frac{-(3-2p)}{2(p-1)}}\left(\ln \frac{n}{i}\right)^{1/2}+t\left(\sum_{i=1}^{n/e} i^{\frac{-2r}{p-1}} \beta_i^{\frac{-(3-2p)}{p-1}} \right)^{1/2}
\]
over the collection of all $\beta\in \mathbb{R}^{\left\lfloor n/e \right\rfloor}$ such that $i^{\frac{2r}{3-2p}}\beta_i$ is positive and non-decreasing in $i$ and such that $\sum\beta_{i}=1$. The method of Lagrange multipliers leads us to solve the equations
\[
\frac{\partial \psi\left(\beta\right)}{\partial \beta_i}=-\lambda
\]
which can be written as
\[
B_1i^{\frac{-r}{p-1}}\left( \ln \frac{n}{i}\right)^{1/2}\beta_i^{\frac{-1}{2(p-1)}}+B_2i^{\frac{-2r}{p-1}}\beta_i^{\frac{-(2-p)}{p-1}}=1
\]
where $B_1$ and $B_2$ are positive values that do not depend on $i$. This implies that
\[
1/2 \leq \max \left\{B_1i^{\frac{-r}{p-1}}\left( \ln \frac{n}{i}\right)^{1/2}\beta_i^{\frac{-1}{2(p-1)}},B_2i^{\frac{-2r}{p-1}}\beta_i^{\frac{-(2-p)}{p-1}}\right\}\leq 1
\]
and therefore
\[
\beta_i\leq \max \left\{2^{2(p-1)}B_1^{2(p-1)}i^{-2r}\left( \ln \frac{n}{i}\right)^{p-1},2^{\frac{p-1}{2-p}}B_2^{\frac{p-1}{2-p}}i^{\frac{-2r}{2-p}} \right\}
\]
with the reverse inequality holding when $2^{2(p-1)}$ and $2^{\frac{p-1}{2-p}}$ are deleted. At this point, and by homogeneity, we remove the condition $\sum \beta_i=1$ and are led to the definition
\[
\beta_i=Ai^{-2r}\left( \ln \frac{n}{i}\right)^{p-1}+i^{-1}
\]
for some $A>0$. $B_1$, $B_2$ and the powers of $2$ dissapear since they do not depend on $i$ and we have re-scaled $\beta$, and we have used the equation $p=2(1-r)$ to simplify the exponent $-2r/(2-p)$. We now minimize over $A$. It follows from Lemma \ref{Lo basic sum bound} that
\begin{eqnarray*}
\sum_{i=1}^{n/e}\beta_i &\leq& CA(1-2r)^{-p}n^{1-2r}+C\ln n
\end{eqnarray*}
With an eye on (\ref{essential bound}), it is clear that the bounds for $\mathbb{M} \left\vert X \right\vert_{\sharp}$ and $\mathrm{Lip}\left( \left\vert \cdot \right\vert_{\sharp}\right)$ are decreasing in $A$. It therefore does not help to let $A$ slip below the point where
\[
CA(1-2r)^{-p}n^{1-2r}=C\ln n
\]
because as $A$ continues to decrease beyond this point $\sum \beta_i$ stays the same order of magnitude while $\mathbb{M} \left\vert X \right\vert_{\sharp}+t\mathrm{Lip}\left( \left\vert \cdot \right\vert_{\sharp}\right)$ increases. We may therefore assume that
\[
A\geq \frac{c(1-2r)^{p}\ln n}{n^{1-2r}} \hspace{3cm} \sum_{i=1}^{n/e}\beta_i \leq CA(1-2r)^{-p}n^{1-2r}
\]
If we look back at (\ref{essential bound}) with our new bound for $\sum \beta_i$ and our definition of $\beta_i$, and we take $A$ out of the expression for $\sum \beta_i$ and move it into the powers of $\beta_i$ with corresponding exponents $\frac{-(3-2p)}{2(p-1)}$ and $\frac{-(3-2p)}{p-1}$ in the expression $\mathbb{M} \left\Vert X \right\Vert +t \mathrm{Lip}\left( \left\Vert \cdot \right\Vert\right)$, we see that these powers of $\beta_i$ become
\[
\left( i^{-2r}\left( \ln \frac{n}{i}\right)^{p-1}+A^{-1}i^{-1} \right)^{\frac{-(3-2p)}{2(p-1)}} \hspace{2cm} \left( i^{-2r}\left( \ln \frac{n}{i}\right)^{p-1}+A^{-1}i^{-1} \right)^{\frac{-(3-2p)}{p-1}}
\]
So, in our current range for $A$, the expression to be minimized (or at least the bound that we have for it) is increasing. This leads us to take
\begin{equation}
A= \frac{(1-2r)^{p}\ln n}{n^{1-2r}} \label{Adef}
\end{equation}
Recall that for $\left\vert \cdot \right\vert_{\sharp}$ to be a norm, see (\ref{normee}), it is sufficient for $i^{2r/(3-2p)}\beta_i$ to be non-decreasing in $i$, equivalently for $i^{-r/(p-1)}\beta_{i}^{-(3-2p)/(2p-2)}$ to be non-increasing. Writing 
\[
\omega_i=i^{\frac{-r}{p-1}}\beta_{i}^{\frac{-(3-2p)}{2(p-1)}}=\left[An^{\frac{4r(p-1)}{3-2p}}\left(\frac{n}{i} \left( \ln \frac{n}{i} \right)^{\frac{-(3-2p)}{4r}}\right)^{\frac{-4r(p-1)}{3-2p}}+i^{\frac{p-1}{3-2p}}\right]^{\frac{-(3-2p)}{2(p-1)}}
\]
and noting that $z\left( \ln z\right)^{-(3-2p)/(4r)}$ is increasing for $z\geq \exp \left((3-2p)/(4r)\right)$, we see that $\omega_i$ is decreasing. We now bound $\mathbb{M} \left\vert X \right\vert_{\sharp}$ and $\mathrm{Lip}\left( \left\vert \cdot \right\vert_{\sharp}\right)$. From the definition of $\beta_i$,
\begin{equation}
\beta_i^{-1} \leq \min\left\{A^{-1}i^{2r}\left(\ln \frac{n}{i}  \right)^{-(p-1)}, i\right\} \label{beta inv}
\end{equation}
which leads us to solve,
\begin{eqnarray*}
A^{-1}i^{2r}\left(\ln \frac{n}{i}  \right)^{-(p-1)}=i
\end{eqnarray*}
Keeping in mind that $1-2r=p-1$, the above equation holds precisely when
\begin{equation}
\frac{n}{i} \left(\ln \frac{n}{i}\right)^{-1}=A^{\frac{1}{p-1}}n \label{eoi}
\end{equation}
The function $z\mapsto z/\ln z$ is increasing on $[e,\infty)$ and we will show that provided $n>n_0$ (for a universal constant $n_0>1$),
\begin{equation}
\frac{1}{2}e^{2}\leq A^{\frac{1}{p-1}}n \leq \frac{2e}{3\ln n}n^{\frac{3}{2e}} \label{rangee}
\end{equation}
so that (\ref{eoi}) has exactly one solution for $i\in [n^{1-\frac{3}{2e}},n/e^2]$, denoted $A_0$ (not necessarily an integer) which satisfies
\begin{eqnarray}
n^{1-\frac{3}{2e}}&\leq& A_0\leq e^{-2}n\\
A_0 \ln \frac{n}{A_0}&=&A^{\frac{-1}{p-1}} \label{basic AA0}\\
\ln \left(A^{1/(p-1)}n \right)&=& \ln \frac{n}{A_0} - \ln \ln \frac{n}{A_0}\\
\ln \left(A^{1/(p-1)}n \right) &\leq& \ln \left( \frac{n}{A_0}\right) \leq \left( 1-\frac{1}{e}\right) \ln \left(A^{1/(p-1)}n \right) \label{logsall}
\end{eqnarray}
The assumption $n>n_0$ does not limit our generality since the result is directly seen to hold when $n\leq n_0$ in which case many of the coefficients involved are bounded by constants. From (\ref{logsall}) and the defining inequality of the current sub-case, it follows that
\[
\frac{1}{2}\ln \left((1-2r)\ln n \right) \leq (1-2r)\ln \left( \frac{n}{A_0}\right) \leq \left( 1-\frac{1}{e}\right) \ln \left((1-2r)\ln n \right)
\]
For the left inequality we used the fact that $p-1\in (0,1/2)$ and $z\ln z \geq -1/e$ for $z\in (0,1/2)$. The right inequality is more straightforward. We now verify (\ref{rangee}). Recalling (\ref{Adef}) and the fact that $1-2r=p-1$, which we are using constantly, the lower bound in (\ref{rangee}) holds provided
\[
\ln n \geq \frac{e^{(2-\ln 2)(p-1)}}{(p-1)^p}
\]
From the definition of the current sub-case, $\ln n \geq e/(p-1)$, so a sufficient condition for the above inequality to hold is
\[
2-\ln2 \leq \frac{1}{p-1}+\ln (p-1)
\]
which is true by considering $1/z+\ln z$ for $z\in \left( 0,1/2 \right)$. The upper bound in (\ref{rangee}) holds provided
\[
(p-1)\ln n \leq \exp\left( \frac{p-1}{p} \ln \left(\frac{2e}{3}n^{\frac{3}{2e}}\right)\right)
\]
which holds by applying $e^z\geq ez$. This completes the task of verifying (\ref{rangee}). From (\ref{beta inv}) it follows that
\begin{equation}
\beta_i^{-1}\leq
\left\{ 
\begin{array}{ccc}
Ci & : & i\leq A_0 \\ 
CA^{-1}i^{2r}\left(\ln \frac{n}{i}\right)^{-(p-1)} & : & i> A_0%
\end{array}%
\right. \label{piecewise beta}
\end{equation}
In an integral where the integrand grows or decays at a controlled rate, one can change an upper bound of $A_0+1$ to $A_0$ at the expense of a constant. Using
\[
 \int_{a}^{b}e^{-\omega}\omega^{1/2}d\omega \leq \frac{C(b-a)e^{-a}a^{1/2}}{1+b-a}
\]
valid as long as $1\leq a \leq b$, and
\[
\int_{0}^{b}e^{-\omega}\omega^{p-1}d\omega \leq C
\]
which gives the correct order of magnitude for (say) $b\geq 1/2$, $\mathbb{M}\left\vert X \right\vert_{\sharp}$ is bounded above by
\begin{eqnarray*}
&&C^{\frac{1}{p-1}}\sum_{i=1}^{A_0}i^{-\frac{1}{2}}\left( \ln \frac{n}{i} \right)^{\frac{1}{2}}+C^{\frac{1}{p-1}}A^{\frac{-(3-2p)}{2(p-1)}}\sum_{i=A_0}^{n/e}i^{-2r}\left( \ln \frac{n}{i} \right)^{p-1}\\
&\leq& C^{\frac{1}{p-1}}n^{-\frac{1}{2}}\int_1^{A_0}\left( \frac{n}{x}\right)^{1/2}\left( \ln \frac{n}{x}\right)^{1/2}dx+C^{\frac{1}{p-1}}A^{\frac{-(3-2p)}{2(p-1)}}n^{-2r}\int_{A_0}^{n/e}\left( \frac{n}{x}\right)^{2r}\left( \ln \frac{n}{x}\right)^{p-1}dx\\
&\leq& C^{\frac{1}{p-1}}n^{\frac{1}{2}}\int_{\frac{1}{2}\ln \frac{n}{A_0}}^{\frac{1}{2}\ln n}e^{-\omega}\omega^{1/2}d\omega +C^{\frac{1}{p-1}}A^{\frac{-(3-2p)}{2(p-1)}}(1-2r)^{-p}n^{1-2r}\int_{1-2r}^{(1-2r)\ln \frac{n}{A_0}}e^{-\omega}\omega^{p-1}d\omega\\
&\leq& C^{\frac{1}{p-1}}A_0^{\frac{1}{2}}\left(\ln \frac{n}{A_0}\right)^{\frac{1}{2}}+C^{\frac{1}{p-1}}A^{\frac{-(3-2p)}{2(p-1)}}(1-2r)^{-p}n^{1-2r}\\
&\leq& C^{\frac{1}{p-1}}A^{\frac{-(3-2p)}{2(p-1)}}(1-2r)^{-p}n^{1-2r}
\end{eqnarray*}
We claim that
\begin{eqnarray*}
\int_0^be^{\omega}\omega^{-(3-2p)}d\omega \leq
\left\{ 
\begin{array}{ccc}
C(p-1)^{-1}b^{2(p-1)} & : & 0\leq b\leq 1 \\ 
C(p-1)^{-1}+Ce^bb^{-(3-2p)} & : & b \geq 1%
\end{array}%
\right.
\end{eqnarray*}
For $0\leq b \leq 1$ this is clear. For $b\geq3$ this follows because on $[2,\infty)$ the local exponential growth rate of the integrand is
\[
\frac{d}{d\omega}\left[\omega-(3-2p)\ln \omega\right]=1-\frac{3-2p}{\omega}\in\left[0.5,1\right]
\]
and for $1<b<3$ the bound follows by monotonicity in $b$. Using the claim just proved, $\textrm{Lip}\left(\left\vert \cdot \right\vert_{\sharp} \right)$ is bounded above by
\begin{eqnarray*}
&&\left[C^{\frac{1}{p-1}}\ln A_0+C^{\frac{1}{p-1}}A^{\frac{-(3-2p)}{p-1}}n^{-4r}\int_{A_0}^{n/e}\left( \frac{n}{x}\right)^{4r}\left( \ln \frac{n}{x}\right)^{-(3-2p)}dx \right]^{\frac{1}{2}}\\
&\leq&C^{\frac{1}{p-1}}\left(\ln A_0 \right)^{\frac{1}{2}}+C^{\frac{1}{p-1}}A^{\frac{-(3-2p)}{2(p-1)}}(4r-1)^{1-p}n^{\frac{1-4r}{2}}\left(\int_{4r-1}^{(4r-1)\ln \frac{n}{A_0}}e^{\omega}\omega^{-(3-2p)}d\omega \right)^{\frac{1}{2}}
\end{eqnarray*}
If $(4r-1)\ln \left(n/A_0\right)<1$ then this is bounded by
\[
C^{\frac{1}{p-1}}\left(\ln n \right)^{\frac{1}{2}}+C^{\frac{1}{p-1}}A^{\frac{-(3-2p)}{2(p-1)}}n^{\frac{1-4r}{2}}\left(\ln \frac{n}{A_0}\right)^{p-1}\leq C^{\frac{1}{p-1}}\left(\ln n \right)^{\frac{1}{2}}
\]
To see why the last inequality is true, note that the inequality
\[
\left(\ln n \right)^{\frac{1}{2}}\geq A^{\frac{-(3-2p)}{2(p-1)}}n^{\frac{1-4r}{2}}\left(\ln \frac{n}{A_0}\right)^{p-1}
\]
reduces to
\[
A_0^{\frac{3-2p}{2}}\left(\ln \frac{n}{A_0}\right)^{\frac{1}{2}}\leq n^{\frac{4r-1}{2}}\left(\ln n\right)^{\frac{1}{2}}
\]
which in turn follows since $1\leq A_0\leq n$ and $3-2p=4r-1$. If $(4r-1)\ln \left(n/A_0\right)\geq1$ then using $4r-1=3-2p$ and $A_0\ln \left(n/A_0\right)=A^{-1/(p-1)}$, we get the same bound, i.e.
\[
\textrm{Lip}\left(\left\vert \cdot \right\vert_{\sharp} \right)\leq C^{\frac{1}{p-1}}\left(\ln n \right)^{\frac{1}{2}}
\]
Going all the way back to (\ref{essential bound}), regardless of whether $(4r-1)\ln \left(n/A_0\right)$ lies in $[0,1)$ or $[1,\infty)$,
\[
\sum_{i=1}^ni^{-2r}X_{[i]}^{2(p-1)}\leq C (1-2r)^{-p}n^{1-2r}+Ct^{2(p-1)}\left( \ln n \right)^{2-p}
\]
Then note that
\[
\frac{(1-2r)^{-p}}{(1-2r)^{-1}}=\exp \left( -(p-1) \ln (p-1) \right) \in (c,1)
\]
\end{proof}

\section{Statement and proof of the main result}

The following diagram indicates the various cases considered in Theorem \ref{LoMain}; it will be useful to refer back to it when reading the proof.

\bigskip

\begin{tikzpicture}
\fill[black!08!white] (0,4.5) -- (1.5,4.5) -- (1.5,8) -- (0,8) -- cycle;
\fill[black!05!white] (1.5,4.5) -- (3,4.5) -- (3,8) -- (1.5,8) -- cycle;
\fill[black!08!white] (3,4.5) -- (6,4.5) -- (6,8) -- (3,8) -- cycle;
\fill[black!05!white] (0,4.5) -- (0.75,3) -- (1.5,3) -- (1.5,4.5) -- cycle;
\fill[black!08!white] (1.5,3) -- (3,3) -- (3,4.5) -- (1.5,4.5) -- cycle;
\fill[black!05!white] (3,3) -- (6,3) -- (6,4.5) -- (3,4.5) -- cycle;
\fill[black!10!white] (0,3) -- (0.75,3) -- (0,4.5) -- cycle;
\draw[thick,->] (0,0) -- (6.5,0) node[anchor=north west] {$r$};
\draw[thick,->] (0,0) -- (0,7) node[anchor=south east] {$p$};
\draw (0 cm,1pt) -- (0 cm,-1pt) node[anchor=north] {$0$};
\draw (0.75 cm,1pt) -- (0.75 cm,-1pt) node[anchor=north] {$\frac{1}{4}$};
\draw (1.5 cm,1pt) -- (1.5 cm,-1pt) node[anchor=north] {$\frac{1}{2}$};
\draw (3 cm,1pt) -- (3 cm,-1pt) node[anchor=north] {$1$};
\draw (6 cm,1pt) -- (6 cm,-1pt) node[anchor=north] {$2$};
\draw (1pt,3 cm) -- (-1pt,3 cm) node[anchor=east] {$1$};
\draw (1pt,4.5 cm) -- (-1pt,4.5 cm) node[anchor=east] {$\frac{3}{2}$};
\draw (1pt,6 cm) -- (-1pt,6 cm) node[anchor=east] {$2$};
\draw[black!17.5!white,dashed] (0.75,0) -- (0.75,7);
\draw[black!17.5!white,dashed] (1.5,0) -- (1.5,7);
\draw[black!17.5!white,dashed] (3,0) -- (3,3);
\draw[black!20!white,dashed] (0,6) -- (3,0);
\draw[black!20!white,dashed] (0,4.5) -- (2.25,0);
\draw[dashed] (0,3) -- (6,3);
\draw[dashed] (0,4.5) -- (6,4.5);
\draw[dashed] (3,3) -- (3,7.5);
\draw[dashed] (6,3) -- (6,7.5);
\draw[black!12!white,dashed] (0,6) -- (5,6);
\draw[dashed] (1.5,3) -- (1.5,7.5);
\draw[dashed] (0,4.5) -- (0.75,3);
\draw[thick] (0.75,4.5) -- (1.5,3);
\filldraw[black](5.5,-2.5) circle (0pt) node[anchor=west]{Figure 1};
\filldraw[black](0.5,6) circle (0pt) node[anchor=west]{ia};
\filldraw[black](1.8,6) circle (0pt) node[anchor=west]{ib*};
\filldraw[black](3.8,6) circle (0pt) node[anchor=west]{ib**};
\filldraw[black](0.5,3.75) circle (0pt) node[anchor=west]{iia};
\filldraw[black](1.8,3.75) circle (0pt) node[anchor=west]{iib*};
\filldraw[black](3.8,3.75) circle (0pt) node[anchor=west]{iib**};
\filldraw[black](0,3.2) circle (0pt) node[anchor=west]{iii};
\filldraw[black](1.7,2) circle (0pt) node[anchor=west]{iv};
\draw[black,thick,->] (1.875,2.25) -- (1.575,2.85);
\filldraw[black](6.5,7.5) circle (0pt) node[anchor=west]{Case ia: $\frac{3}{2}\leq p<\infty$, $0\leq r \leq \frac{1}{2}$};
\filldraw[black](6.5,6.5) circle (0pt) node[anchor=west]{Case ib*: $\frac{3}{2}\leq p<\infty$, $\frac{1}{2}< r \leq 1 $};
\filldraw[black](6.5,5.5) circle (0pt) node[anchor=west]{Case ib**: $\frac{3}{2}\leq p<\infty$, $1< r \leq 2 $};
\filldraw[black](6.5,4.5) circle (0pt) node[anchor=west]{Case iia: $1\leq p<\frac{3}{2}$, $\frac{3-2p}{4}\leq r \leq \frac{1}{2}$, $p\neq 2-2r$};
\filldraw[black](6.5,3.5) circle (0pt) node[anchor=west]{Case iib*: $1\leq p<\frac{3}{2}$, $\frac{1}{2}< r \leq 1$};
\filldraw[black](6.5,2.5) circle (0pt) node[anchor=west]{Case iib**: $1\leq p<\frac{3}{2}$, $1< r \leq 2$};
\filldraw[black](6.5,1.5) circle (0pt) node[anchor=west]{Case iii: $1\leq p<\frac{3}{2}$, $p< \frac{3}{2}-2r$};
\filldraw[black](6.5,0.5) circle (0pt) node[anchor=west]{Case iv: $1\leq p<\frac{3}{2}$, $p=2-2r$};

\draw[black,dashed] (0,-1.5) -- (1.2,-1.5);
\filldraw[black](1.2,-1.5) circle (0pt) node[anchor=west]{: boundary of a region};
\draw[black!25!white,dashed] (6.6,-1.55) -- (7.8,-1.5);
\filldraw[black](7.8,-1.5) circle (0pt) node[anchor=west]{: other relevant line};
\end{tikzpicture}

\begin{theorem}
\label{LoMain}There exist universal constants $C,c>0$ and a function $\left(r,p\right)\mapsto c_{r,p}$ from $ [0,2]\times [1,\infty)$ to $(0,\infty)$ such that the
following is true. Let%
\[
\left( n,k,\varepsilon ,r,p\right) \in \mathbb{N}\times \mathbb{N}\times
\left( 0,1/2\right) \times \left[0,2\right] \times \left[ 1,\infty
\right) 
\]%
and let $G$ be a random $n\times k$ matrix with i.i.d. standard normal
random variables as entries. Cases i-iv will be defined as in Figure 1 above. In each case variables $E$ and $F$ will be defined, and as long as $k\leq \min \left\{E,F\right\} $, with probability at least $1-C\exp \left(
-\min \left\{E,F\right\} \right) $ the following event
occurs: for all $x\in \mathbb{R}^{k}$,%
\[
\left( 1-\varepsilon \right) M_{r,p}\left\vert x\right\vert \leq \left\vert
Gx\right\vert _{r,p}\leq \left( 1+\varepsilon \right) M_{r,p}\left\vert
x\right\vert 
\]%
where $M_{r,p}$ denotes the median of $\left\vert Ge_{1}\right\vert _{r,p}$. $E$ and $F$ are defined as follows:

\noindent  In \underline{Case ia}: $\frac{3}{2}\leq p<\infty$, $0\leq r \leq \frac{1}{2}$ and
 \begin{eqnarray*}
E&=&\frac{c^pn\left(\ln n \right)^2\left[p+(1-2r)\ln n\right]^p\varepsilon^2}{\left(p+\ln n\right)^{2+p}}
\geq\left\{ 
\begin{array}{ccc}
c_{r,p}n\varepsilon^2 & : & r\neq 1/2 \\ 
c_{r,p}n\left(\ln n\right)^{-p}\varepsilon^2& : & r=1/2
\end{array}%
\right. \\
F&=&\frac{cpn^{\frac{2(1-r)}{p}}\left(\ln n\right)^{1+\frac{2}{p}}\varepsilon^{\frac{2}{p}}}{\left(1+n^{\frac{2-2r-p}{p}}\right)\left(p+\ln n\right)^{1+\frac{2}{p}}}\left(\frac{1+\left\vert 2-2r-p \right\vert \ln n}{\ln n}\right)^{\max\left\{\frac{2-p}{p},0\right\}}\\
&\geq&\left\{
\begin{array}{ccc}
c_{r,p}n\varepsilon^{\frac{2}{p}} & : & p<2-2r \\ 
c_{r,p}n\left(\ln n\right)^{-\left(\frac{2}{p}-1\right)}\varepsilon^{\frac{2}{p}} & : & p=2-2r \\ 
c_{r,p}n^{\frac{2(1-r)}{p}}\varepsilon^{\frac{2}{p}} & : & p>2-2r
\end{array}%
\right.
 \end{eqnarray*}
In \underline{Case ib*}: $\frac{3}{2}\leq p<\infty$, $\frac{1}{2}< r \leq 1 $ and
 \begin{eqnarray*}
E&=&\frac{c^pp^pn^{2(1-r)}\left(\ln n\right)^2\left[1+(2r-1)\ln n \right]\varepsilon^2}{\left[p+(1-r)\ln n \right]^{2+p}}\geq\left\{ 
\begin{array}{ccc}
c_{r,p}n^{2(1-r)}\left(\ln n\right)^{-(p-1)}\varepsilon^2 & : & r\neq 1 \\ 
c_{r,p}n^{2(1-r)}\left(\ln n\right)^{3}\varepsilon^2& : & r=1
\end{array}%
\right.\\
F&=&\frac{cpn^{\frac{2(1-r)}{p}}\left(\ln n\right)^{1+\frac{2}{p}}\varepsilon^{\frac{2}{p}}}{\left[p+(1-r)\ln n\right]^{1+\frac{2}{p}}}\geq\left\{ 
\begin{array}{ccc}
c_{r,p}n^{\frac{2(1-r)}{p}}\varepsilon^{\frac{2}{p}} & : & r\neq 1 \\ 
c_{r,p}n^{\frac{2(1-r)}{p}}\left(\ln n\right)^{1+\frac{2}{p}}\varepsilon^{\frac{2}{p}}& : & r=1
\end{array}%
\right.
 \end{eqnarray*}
In \underline{Case ib**}: $\frac{3}{2}\leq p<\infty$, $1< r \leq 2 $ and
\begin{eqnarray*}
E&=&\frac{c^p\left(\ln n\right)^3\varepsilon^2}{\left[1+(r-1)\ln n\right]^2}\geq\left\{ 
\begin{array}{ccc}
c_{r,p}\left(\ln n\right)\varepsilon^2 & : & r\neq 1 \\ 
c_{r,p}\left(\ln n\right)^{3}\varepsilon^2& : & r=1
\end{array}%
\right.\\
F&=&\frac{c\left(\ln n\right)^{1+\frac{2}{p}}\varepsilon^{\frac{2}{p}}}{\left[1+(r-1)\ln n\right]^\frac{2}{p}}\geq\left\{ 
\begin{array}{ccc}
c_{r,p}\left(\ln n\right)\varepsilon^{\frac{2}{p}} & : & r\neq 1 \\ 
c_{r,p}\left(\ln n\right)^{1+\frac{2}{p}}\varepsilon^{\frac{2}{p}}& : & r=1
\end{array}%
\right.
 \end{eqnarray*}
In \underline{Case iia}: $1\leq p<\frac{3}{2}$, $\frac{3-2p}{4}\leq r \leq \frac{1}{2}$, $p\neq 2-2r$ and
\begin{eqnarray*}
E&=&\frac{cn\left[1+(1-2r)\ln n\right]^p\varepsilon^2}{\left(\ln n\right)^p}\geq\left\{ 
\begin{array}{ccc}
c_{r,p}n\varepsilon^2 & : & r\neq 1/2 \\ 
c_{r,p}n\left(\ln n\right)^{-p}\varepsilon^2& : & r=1/2
\end{array}%
\right.\\
F&=&\frac{cn^{\frac{2(1-r)}{p}}\varepsilon^{\frac{2}{p}}}{1+n^{\frac{2-2r-p}{p}}}\left(\frac{1+\left\vert2-2r-p\right\vert \ln n}{\ln n}\right)^{1/p}\geq\left\{
\begin{array}{ccc}
c_{r,p}n\varepsilon^{\frac{2}{p}} & : & p<2-2r \\ 
c_{r,p}n\left(\ln n\right)^{-\frac{1}{p}}\varepsilon^{\frac{2}{p}} & : & p=2-2r \\ 
c_{r,p}n^{\frac{2(1-r)}{p}}\varepsilon^{\frac{2}{p}} & : & p>2-2r
\end{array}%
\right.
\end{eqnarray*}
 In \underline{Case iib*}: $1\leq p<\frac{3}{2}$, $\frac{1}{2}< r \leq 1$ and
 \begin{eqnarray*}
E&=&\frac{cn^{2(1-r)}\left(\ln n\right)^{2}\left[1+(2r-1)\ln n\right]\varepsilon^2}{\left[1+(1-r)\ln n\right]^{2+p}}\geq\left\{ 
\begin{array}{ccc}
c_{r,p}n^{2(1-r)}\left(\ln n\right)^{-(p-1)}\varepsilon^{2} & : & r\neq 1 \\ 
c_{r,p}n^{2(1-r)}\left(\ln n\right)^{3}\varepsilon^{2}& : & r=1
\end{array}%
\right.\\
F&=&\frac{cn^{\frac{2(1-r)}{p}}\left(\ln n\right)^{1+\frac{1}{p}}\left[1+\left\vert2-2r-p\right\vert \ln n\right]^{\frac{1}{p}}\varepsilon^{\frac{2}{p}}}{\left[1+(1-r)\ln n\right]^{1+\frac{2}{p}}}
\geq\left\{ 
\begin{array}{ccc}
c_{r,p}n^{\frac{2(1-r)}{p}}\varepsilon^{\frac{2}{p}} & : & r\neq 1 \\ 
c_{r,p}n^{\frac{2(1-r)}{p}}\left(\ln n\right)^{1+\frac{2}{p}}\varepsilon^{\frac{2}{p}} & : & r=1
\end{array}%
\right.
 \end{eqnarray*}
In \underline{Case iib**}: $1\leq p<\frac{3}{2}$, $1< r \leq 2$ and
\begin{eqnarray*}
E&=&\frac{c\left(\ln n\right)^3\varepsilon^2}{\left[1+(r-1)\ln n\right]^2}\geq c_{r,p}\left(\ln n\right)\varepsilon^2 \\
F&=&\frac{c\left(\ln n\right)^{1+\frac{2}{p}}\varepsilon^{\frac{2}{p}}}{\left[1+(r-1)\ln n\right]^\frac{2}{p}}
\geq c_{r,p}\left(\ln n\right)\varepsilon^{\frac{2}{p}}
\end{eqnarray*}
 In \underline{Case iii}: $1\leq p<\frac{3}{2}$, $p< \frac{3}{2}-2r$ and
\begin{eqnarray*}
E&=&cn\varepsilon^2\\
F&=&\frac{cn\left[1+(1-4r)\ln n\right]^{1-\frac{1}{p}}\varepsilon^{\frac{2}{p}}}{\left(\ln n\right)^{1-\frac{1}{p}}}\geq c_{r,p}n\varepsilon^{\frac{2}{p}}
\end{eqnarray*}
In \underline{Case iv}: $1\leq p<\frac{3}{2}$, $p=2-2r$ and
\begin{eqnarray*}
E&=&\frac{cn\left[1+(1-2r)\ln n\right]\varepsilon^2}{\ln n}
\geq\left\{ 
\begin{array}{ccc}
c_{r,p}n\varepsilon^2 & : & r\neq 1/2 \\ 
c_{r,p}n\left(\ln n\right)^{-1}\varepsilon^2 & : & r=1/2
\end{array}%
\right.\\
F&=&\frac{cn\varepsilon^{\frac{2}{p}}}{\left(\ln n\right)^{\frac{2-p}{p}}}
 \end{eqnarray*}    
\end{theorem}

\bigskip

The parameters $E$ and $F$ in Theorem \ref{LoMain} can be bounded as follows:
\[
E\geq \left\{ 
\begin{array}{ccc}
c_{r,p}n\varepsilon^2 & : & 0\leq r<1/2\\ 
c_{r,p}n\left(\ln n\right)^{-p}\varepsilon^2 & : & r=1/2\\
c_{r,p}n^{2(1-r)}\left(\ln n\right)^{-(p-1)}\varepsilon^2 & : & 1/2<r<1\\
c_{r,p}\left(\ln n\right)^{3}\varepsilon^2 & : & r=1
\end{array}%
\right.
\]
and
\[
F\geq \left\{ 
\begin{array}{ccc}
c_{r,p}n\varepsilon^{\frac{2}{p}} & : & 0\leq r\leq1/2, p<2-2r\\ 
c_{r,p}n\left(\ln n\right)^{1-\frac{2}{p}}\varepsilon^{\frac{2}{p}} & : & 0\leq r\leq1/2, p=2-2r\\
c_{r,p}n^{\frac{2(1-r)}{p}}\varepsilon^{\frac{2}{p}} & : & 0\leq r\leq1/2, p>2-2r\\
c_{r,p}n^{\frac{2(1-r)}{p}}\varepsilon^{\frac{2}{p}} & : & 1/2<r<1\\
c_{r,p}\left(\ln n\right)^{1+\frac{2}{p}}\varepsilon^{\frac{2}{p}} & : & r=1
\end{array}%
\right.
\]
For $n\geq n_0(r,p)$, the coefficient $c_{r,p}$ can be written explicitly in terms of $r$ and $p$, as can $n_0(r,p)$.

\bigskip

\begin{proof}
Let $G$ be a random matrix with i.i.d. standard normal random variables as entries, and let $\theta \in S^{n-1}$. $G \theta$ therefore has the standard normal distribution in $\mathbb{R}^n$. Setting $\psi (x)=\sum i^{-r}x_{\left[ i\right] }^{p}$,%
\begin{equation}
\left\vert \nabla \psi (x)\right\vert =p\left( \sum_{i=1}^{n}i^{-2r}x_{\left[
i\right] }^{2(p-1)}\right) ^{1/2}  \label{Lo grad}
\end{equation}%
which is valid for all $x$ with distinct non-zero coordinates. Fix any $t>0$. With (\ref{Lo grad}) in mind,
for $j\in \left\{ 0,1\right\} $ set%
\begin{eqnarray*}
A_{j} &=&\left\{ x\in \mathbb{R}^{n}:\left\vert x \right\vert_{\sharp}\leq \left(\frac{4}{3}\right)^jS\right\}
\end{eqnarray*}%
where $\left\vert \cdot \right\vert_{\sharp}$ and $S$ (and $R$ below) are as in Theorem \ref{orderorderbound}. The cases in that theorem overlap, which is not a problem as long as you pick a case that applies to the values of $p$ and $r$ in question, and stick with that case. We shall apply the cases as follows:

\medskip
$\bullet$ If $3/2\leq p <\infty$ use Case I.

$\bullet$ If $1\leq p <3/2$ and $p\geq 3/2-2r$ and $p\neq 2-2r$ use Case II.

$\bullet$ If $1\leq p <3/2$ and $p< 3/2-2r$ use Case III.

$\bullet$ If $1\leq p <3/2$ and $p=2-2r$ use Case IV (either IVa or IVb, whichever applies).
\medskip
 
\noindent Since $r\leq2$, the bounds in Theorem \ref{orderorderbound} simplify slightly, and can be written as follows. In each of the cases below, definitions are given for $A$ and $B$, and in each case $R\leq A+Bt^{2(p-1)}$, where $R$ is as defined in Theorem \ref{orderorderbound} (with the same value of $t$ that appears here). Conditions defining the case (i.e. i, ii, iii, iv) come first and are either without brackets or with square brackets [...], the square brackets indicating a redundant condition. Conditions defining the sub-case (i.e. a, b) come last and are in parentheses (...).

\medskip

\textbf{Case  ia:} $3/2\leq p <\infty$ (and $0\leq r\leq1/2$),
\begin{eqnarray*}
A&=&\frac{C^p p^p n^{1-2r} \left( \ln n \right)^p}{\left[ p+(1-2r) \ln n\right]^p}\\
B&=&C^p\left( \frac{\ln n}{1+\left\vert 2-2r-p \right\vert \ln n} \right)^{\max\{2-p,0\}}\left( 1+n^{2-2r-p} \right)
\end{eqnarray*}

\medskip

\textbf{Case  ib:} $3/2\leq p <\infty$ (and $1/2< r\leq2$),
\begin{eqnarray*}
A&=&\frac{C^p\left( \ln n \right)^p}{1+(2r-1) \ln n }\\
B&=&C^p\left( \frac{\ln n}{1+\left\vert 2-2r-p \right\vert \ln n} \right)^{\max\{2-p,0\}}\left( 1+n^{2-2r-p} \right)
\end{eqnarray*}

\medskip

\textbf{Case  iia:} $1\leq p <3/2$, $p\geq3/2-2r$, $p\neq 2-2r$ (and $0\leq r\leq1/2$),
\begin{eqnarray*}
A&=&\frac{Cn^{1-2r}\left(\ln n\right)^p}{\left[1+(1-2r)\ln n \right]^p}\\
B&=&C\frac{1+n^{2-2r-p}}{1+\left\vert 2-2r-p\right\vert\ln n}\ln n
\end{eqnarray*}

\medskip

\textbf{Case  iib:} $1\leq p <3/2$, $[p\geq3/2-2r]$, $[p\neq 2-2r]$ (and $1/2< r\leq2$),
\begin{eqnarray*}
A&=&C\frac{\left(\ln n \right)^p}{1+(2r-1)\ln n}\\
B&=&C\frac{1+n^{2-2r-p}}{1+\left\vert 2-2r-p\right\vert\ln n}\ln n
\end{eqnarray*}

\medskip

\textbf{Case  iii:} $1\leq p <3/2$ and $p<3/2-2r$,
\begin{eqnarray*}
A&=&Cn^{1-2r}\\
B&=&Cn^{2-2r-p}\left( \frac{\ln n}{1+(1-4r) \ln n} \right)^{p-1}
\end{eqnarray*}

\medskip

\textbf{Case  iv:} $1\leq p <3/2$ and $p=2-2r$,
\begin{eqnarray*}
A&=&C\min\left\{(1-2r)^{-1},\ln n\right\}n^{1-2r}\\
B&=&C\left( \ln n \right)^{2-p}
\end{eqnarray*}

\medskip

By Theorem \ref{orderorderbound}, $\gamma _{n}\left( A_{1}\right) \geq \gamma
_{n}\left( A_{0}\right) \geq 1-C\exp \left( -t^{2}/2\right) $. Since $A_{1}$
is convex, a bound on the gradient transfers directly to a bound on the
Lipschitz constant (using Lemma \ref{diff restrict} to ignore points of non-differentiability), and $Lip\left( \psi |_{A_{1}}\right) \leq C^{p}R^{1/2}$.
We may now extend the restriction $\psi |_{A_{1}}$ to a Lipschitz function $%
\psi ^{\ast }:\mathbb{R}^{n}\rightarrow \mathbb{R}$ such that $Lip\left(
\psi ^{\ast }\right) =Lip\left( \psi |_{A_{1}}\right) $. By a result of
Schechtman (see his comments near the end of the paper), as long as $k\leq
ct^{2}$, with probability at least $1-C\exp \left( -ct^{2}\right) $, the
following event occurs: for all $\theta \in S^{k-1}$,%
\begin{equation}
\left\vert \psi ^{\ast }\left( G\theta \right) -\mathbb{E}\psi ^{\ast
}\left( G\theta \right) \right\vert \leq tLip\left( \psi ^{\ast }\right)
\label{Lo 1st conc}
\end{equation}%
Here it is essential to have a result that applies to Lipschitz functions
besides just norms (the most typical application of Schechtman's result is
to norms). We now show that for all $\theta \in S^{k-1}$, $G\theta \in A_{1}$
and therefore $\psi ^{\ast }\left( G\theta \right) =\psi \left( G\theta
\right) $. To do this, consider a $1/4$-net $\mathcal{N}\subset S^{k-1}$
with cardinality $\left\vert \mathcal{N}\right\vert \leq 12^{k}$. By the
union bound, with probability at least $1-12^{k}C\exp \left( -ct^{2}\right)
\geq 1-C\exp \left( -c_{2}t^{2}\right) $, for all $\omega \in \mathcal{N}$, $%
G\omega \in A_{0}$.
Now for any $\theta \in S^{k-1}$ write $\theta =\sum_{i=0}^{\infty
}\varepsilon _{i}\omega _{i}$, where $0\leq \varepsilon _{i}\leq 4^{-i}$ and 
$\omega _{i}\in \mathcal{N}$. Then%
\[
\left\vert G\theta \right\vert _{\sharp }\leq \sum_{i=0}^{\infty
}4^{-i}\left\vert G\omega _{i}\right\vert _{\sharp }\leq \frac{4}{3}S
\]%
which shows that $G\theta \in A_{1}$. (\ref{Lo 1st conc}) can then be
written as%
\[
\left\vert \psi \left( G\theta \right) -\mathbb{E}\psi ^{\ast }\left(
G\theta \right) \right\vert \leq tC^{p}R^{1/2} 
\]%
valid for all $\theta \in S^{k-1}$. It is an elementary calculation that
concentration about any point implies concentration about the median, with
slightly modified constant, so%
\begin{equation}
\left\vert \frac{\psi \left( G\theta \right) }{\mathbb{M}\psi \left( G\theta
\right) }-1\right\vert \leq \frac{C^{p}tR^{1/2}}{\mathbb{M}\psi \left(
G\theta \right) }  \label{Lo first frac}
\end{equation}%
Choose $t$ so that%
\begin{equation}
4p^{-1}\frac{C^{p}tR^{1/2}}{\mathbb{M}\psi \left( G\theta \right) }%
=\varepsilon  \label{Lo t in terms e}
\end{equation}%
and assume that $\varepsilon \in \left( 0,2/p\right) $ so that $p\varepsilon
/4<1/2$ (in order to satisfy $\left\vert 1-u\right\vert \leq 1/2$ below).
The bounds for $\varepsilon \in \left[ 2/p,1/2\right) $ will follow from the
bounds for $\varepsilon \in \left( 0,2/p\right) $ by monotonicity and by
changing the value of $C$ that appears in $C^{p}$. Using $\left\vert
1-u^{1/p}\right\vert \leq p^{-1}2^{1-1/p}\left\vert 1-u\right\vert $, which
holds when $\left\vert 1-u\right\vert \leq 1/2$, with $u=\psi \left( G\theta
\right) /\mathbb{M}\psi \left( G\theta \right) $, (\ref{Lo first frac})
implies%
\[
\left\vert \frac{\psi \left( G\theta \right) ^{1/p}}{\mathbb{M}\psi \left(
G\theta \right) ^{1/p}}-1\right\vert \leq 2p^{-1}\frac{C^{p}tR^{1/2}}{%
\mathbb{M}\psi \left( G\theta \right) }<\varepsilon 
\]%
It then follows from homogeneity that%
\[
\left( 1-\varepsilon \right) \left\vert x\right\vert \mathbb{M}\left\vert
Ge_{1}\right\vert _{r,p}\leq \left\vert Gx\right\vert _{r,p}\leq \left(
1+\varepsilon \right) \left\vert x\right\vert \mathbb{M}\left\vert
Ge_{1}\right\vert _{r,p} 
\]%
for all $x\in \mathbb{R}^{k}$. (\ref{Lo t in terms e}) can be used to bound $%
t$ in terms of $\varepsilon $, since $R$ has been expressed in terms of $t$ in Theorem \ref{orderorderbound},
and in Lemma \ref{Lo med calculation} $\mathbb{M}\psi \left( G\theta \right) 
$ is bounded below in $t$. The sufficient condition $k\leq ct^{2}$ and the
probability bound $1-C\exp \left( -ct^{2}\right) $ can then be written in
terms of $\varepsilon $. The inversion can be simplified by converting a sum
to a $\max $ and using the fact that if $f(x)=$ $\max \left\{
g(x),h(x)\right\} $ with $g,h$ continuous and increasing, then $%
f^{-1}(x)=\min \left\{ g^{-1}(x),h^{-1}(x)\right\} $. The result is that,
\[
t^2\geq \min\left\{c^pA^{-1}\left(\mathbb{M}\sum_{i=1}^n i^{-r}X_{\left[ i\right] }^{p} \right)^2\varepsilon^2, cB^{-1/p}\left(\mathbb{M}\sum_{i=1}^n i^{-r}X_{\left[ i\right] }^{p} \right)^{2/p}\varepsilon^{2/p}\right\}
\]
where $A$ and $B$ are as defined in cases i-iv above, and $X=Ge_1$ follows the standard normal distribution in $\mathbb{R}^n$.
Cases ib and iib split into cases ib*, ib**, iib* and iib**, depending on whether $1/2<r\leq1$ or $1<r\leq2$. The final bounds can then be written as in the statement of the theorem.
\end{proof}

\section{\label{LoGen}The general case}

Here we study the norm
\[
\left\vert x\right\vert_{\omega,p}=\left(\sum_{i=1}^n\omega_i x_{[i]}^p\right)^{1/p}
\]
where $\left(\omega_i\right)_1^n$ is any non-increasing sequence in $[0,1]$ with $\omega_1=1$. Our main interest is in the case $1\leq p<\infty$, however our proof will force us to consider also $0<p<1$. The proof of Lemma \ref{Lo Lip con} generalizes easily and we see that
\[
\sup \left\{ \left( \sum_{i=1}^{n}\omega_i\theta _{\left[ i\right]
}^{p}\right) ^{1/p}:\theta \in S^{n-1}\right\} =\left\{ 
\begin{array}{ccc}
\left( \sum_{1}^{n}\omega_i^{\frac{2}{2-p}}\right) ^{\frac{2-p}{2p}} & : & p\in \left[
1,2\right) \\ 
1 & : & p\in \left[ 2,\infty \right)%
\end{array}%
\right. 
\]
and generalizing Lemma \ref{Lo med calculation}, for $0<p<\infty$,
\[
\mathbb{M}\left\vert X\right\vert_{\omega,p}\leq C\left( \sum_{i=1}^{n}\omega_i\left(\ln \frac{n}{i}\right)^{p/2}\right) ^{1/p}
\]
with the reverse inequality with $C$ replaced by $c$. For $1\leq p<\infty$, $\left\vert \cdot\right\vert_{\omega,p}$ is a norm, while for $0<p<1$ the triangle inequality is replaced with
\[
\left\vert x+y\right\vert_{\omega,p}\leq 2^{1/p}\left(\left\vert x\right\vert_{\omega,p}+\left\vert y\right\vert_{\omega,p}\right)
\]
(proof just as in the classical case of $\ell_p^n$), and $\left\vert \cdot\right\vert_{\omega,p}$ is no longer a norm but a quasi-norm. For an infinite sum, using induction one can show that
\[
\left\vert \sum_{i=1}^\infty x_i\right\vert_{\omega,p}\leq \sum_{i=1}^\infty2^{i/p}\left\vert x_i\right\vert_{\omega,p}
\]

 As sometimes happens, the general result is easier to state and prove than is the special case, since the complexity is hidden. The proof of Theorem \ref{LoMain2} below is a variation of the proof of Theorem \ref{LoMain}, using the quasi-norm property instead of the triangle inequality.

\begin{theorem}
\label{LoMain2}There exists a universal constant $c>0$ such that the
following is true. Let%
\[
\left( n,k,\varepsilon,p\right) \in \mathbb{N}\times \mathbb{N}\times
\left( 0,1/2\right) \times \left[ 1,\infty
\right) 
\]
and let $\left(\omega_i\right)_1^n$ be any non-increasing sequence in $[0,1]$ with $\omega_1=1$. Let $\left\vert \cdot\right\vert_{\omega,p}$ denote the corresponding Lorentz norm, and let $G$ be a random $n\times k$ matrix with i.i.d. standard normal random variables as entries. If $p\neq1$ set
\[
d=\left(1+\frac{1}{p-1}\right)^{-1}\min \left\{\frac{c^p\left( \sum_{i=1}^{n}\omega_i\left(\ln \frac{n}{i}\right)^{p/2}\right) ^{2}\varepsilon^2}{\sum_{i=1}^{n}\omega_i^2\left(\ln \frac{n}{i}\right)^{p-1}},
cB^{-1/p}\left( \sum_{i=1}^{n}\omega_i\left(\ln \frac{n}{i}\right)^{p/2}\right) ^{2/p}\varepsilon^{2/p}\right\}
\]
where
\[
B=\left\{ 
\begin{array}{ccc}
\sum_{i=1}^n\omega_i^2i^{-(p-1)}
 & : & 1< p <3/2\\ 
\left(\sum_{i=1}^n\omega_i^{\frac{2}{2-p}}\right)^{2-p} & : & 3/2\leq p<2\\
1 & : & 2\leq p<\infty
\end{array}%
\right. 
\]
and if $p=1$ set
\[
d=\frac{c\left( \sum_{i=1}^{n}\omega_i\left(\ln \frac{n}{i}\right)^{1/2}\right) ^{2}\varepsilon^2}{\sum_{i=1}^{n}\omega_i^2}
\]
Assume that $k\leq d$. With probability at least $1-2e^{-d} $ the following event
occurs: for all $x\in \mathbb{R}^{k}$,%
\[
\left( 1-\varepsilon \right) \left\vert x\right\vert \mathbb{M}\left\vert
Ge_{1}\right\vert _{\omega,p}\leq \left\vert Gx\right\vert _{\omega,p}\leq \left(
1+\varepsilon \right) \left\vert x\right\vert \mathbb{M}\left\vert
Ge_{1}\right\vert _{\omega,p} 
\]%
\end{theorem}

\begin{proof}
First, assume that $p\neq 1$ (we leave the case $p=1$ to the reader). For $x\in \mathbb{R}^n$, set
\[
\psi\left(x\right)=\sum_{i=1}^n\omega_i x_{[i]}^p \hspace{2cm}
\left\vert x\right\vert_{\sharp}=\left(\sum_{i=1}^n\omega_i^{2}x_{[i]}^{2(p-1)}\right)^{\frac{1}{2(p-1)}}
\]
and so
\[
\left\vert \nabla \psi(x)\right\vert=p\left(\sum_{i=1}^n\omega_i^{2}x_{[i]}^{2(p-1)}\right)^{1/2}
\]
which is valid for all $x$ with distinct non-zero coordinates. The points of non-differentiability are not a problem, by Lemma \ref{diff restrict}. Consider any $\theta \in S^{n-1}$. Set $R=A+C^pBt^{2(p-1)}$, where $A=C^p\sum_{i=1}^{n}\omega_i^2\left(\ln \frac{n}{i}\right)^{p-1}$ and $B$ is defined in the statement of the theorem. For $j\in\{0,1\}$, set
\[
A_j=\left\{x\in\mathbb{R}^n:\left\vert x\right\vert_{\sharp}\leq \left(\frac{4q}{3}\right)^jR^\frac{1}{2(p-1)}\right\}
\]
 where
\[
q=\left\{ 
\begin{array}{ccc}
2^{-1/(p-1)} & : & 1< p <3/2\\ 
1 & : & 3/2\leq p<\infty
\end{array}%
\right.
\]
For $1<p<3/2$ we now use (\ref{orderstat estimate}) from Lemma \ref{normalorderstats} on order statistics from the standard normal distribution, as used in Case II of Theorem \ref{orderorderbound}, and for $3/2\leq p<\infty$ we use Gaussian concentration applied to $\left\vert \cdot\right\vert_{\sharp}$. The result of this is that $\gamma_n\left(A_1\right)\geq \gamma_n\left(A_0\right)\geq 1-C\exp\left(-ct^2\right)$. It follows by definition of $A_1$ and by its convexity, and the expression for $\left\vert \nabla \psi\right\vert$ that $\mathrm{Lip}\left(\psi|_{A_1}\right)\leq C^pR^{1/2}$. By the extension property of real valued Lipschitz functions on metric spaces, $\psi|_{A_1}$ can be extended to a function $\psi^{*}:\mathbb{R}^n\mapsto\mathbb{R}$ such that $\mathrm{Lip}\left(\psi^*\right)=\mathrm{Lip}\left(\psi|_{A_1}\right)$. By Schechtman's result (Theorem \ref{Schechtbound} here), as long as $k\leq ct^2$, the following event occurs with probability at least $1-C\exp\left(-ct^2\right)$: for all $\theta\in S^{k-1}$,
\[
\left\vert \psi^*\left(G\theta\right)- \mathbb{M}\psi^*\left(G\theta\right)\right\vert \leq Ct\mathrm{Lip}\left(\psi^*\right)
\]
We now show that for all $\theta\in S^{k-1}$, $G\theta\in A_1$ and therefore $\psi\left(G\theta\right)=\psi^*\left(G\theta\right)$. Let $\mathcal{N}\subset S^{k-1}$ be a $q/4$-net in $S^{k-1}$. The standard volumetric bound shows that $\mathcal{N}$ can be chosen so that $\left\vert\mathcal{N}\right\vert\leq \left(12/q\right)^{k}$. By the bound on $\gamma_n\left(A_0\right)$ and the union bound, with probability at least $1-C\left(12/q\right)^{k}\exp\left(-ct^2\right)\geq1-C\exp\left(-c't^2\right)$, the following event occurs: for all $\theta\in\mathcal{N}$, $G\theta\in A_0$, i.e.
\[
\left\vert G\theta\right\vert_{\sharp}\leq R^\frac{1}{2(p-1)}
\]
Now for any $\theta \in S^{k-1}$, write $\theta=\sum_1^\infty \varepsilon_i \theta_i$, where $0\leq \varepsilon_i\leq \left(q/4\right)^i$ and $\theta_i\in\mathcal{N}$, to conclude (using the quasi-norm property),
\[
\left\vert G\theta\right\vert_{\sharp}\leq R^\frac{1}{2(p-1)}\sum_{i=0}^\infty\left(q\right)^{-(i-1)}\left(q/4\right)^i\leq \frac{4q}{3}R^\frac{1}{2(p-1)}
\]
So $G\theta\in A_1$ as desired, and $\psi^*\left(G\theta\right)=\psi\left(G\theta\right)$. So, conditioning on the events dealt with above, for all $\theta\in S^{k-1}$,
\[
\left\vert \psi\left(G\theta\right)- \mathbb{M}\psi^*\left(G\theta\right)\right\vert \leq Ct\mathrm{Lip}\left(\psi\right)
\]
The rest of the proof is identical to the proof of Theorem \ref{LoMain}, and we may change the coefficient of $e^{-d}$ from $C$ to 2 by changing the value of $c$.
\end{proof}

The bounds in Theorem \ref{LoMain2} are non-optimal when:

\medskip

\noindent $\bullet$ the coefficient sequence $\left(\omega_i\right)_1^n$ approximates $\left(i^{-r}\right)_1^n$, for $p=2-2r$ ($1/4<r<1/2$), and in this case we refer the reader to Theorem \ref{LoMain} and Corollary \ref{Corrasy},


\noindent $\bullet$ when $r>1$ in which case we refer the reader to Theorem \ref{Lo ell infty regime}, and when


\noindent $\bullet$ $p\rightarrow 1$ ($p\neq 1$), which is due to the presence of the factor $1/(1+1/(p-1))$, but this is not an issue in the asymptotic case $n\rightarrow \infty$ while $\omega$ and $p$ remain fixed. We refer the reader to Theorem \ref{LoMain} and Corollary \ref{Corrp} for estimates that do not include this extra factor.

\end{document}